\documentclass[11pt,leqno]{amsart}   
\usepackage{mathabx}
\usepackage[T1]{fontenc} 
\usepackage{tikz}
\usepackage{cancel}
\usepackage{amsaddr}
\usepackage{amssymb}
\usepackage{amsmath}
\usepackage{amsfonts}
\usepackage{amssymb}
\usepackage{amsthm}
\usepackage{stmaryrd} 
\usepackage{float}
\usepackage{xspace}
\usepackage{epsfig}
\usepackage{xypic}
\usepackage{mathscinet}
\usepackage{fancyhdr}
 \usepackage[normalem]{ ulem } 
\usepackage{soul}

\usepackage[colorlinks=true, linkcolor=blue, 
hyperfootnotes=true,citecolor=blue,urlcolor=black]{hyperref}
 \usepackage{color}

\usepackage{color}
\definecolor{Green}{rgb}{0,.8,0}

\usepackage{graphicx}              

\usepackage{microtype} 

\usepackage[marginparwidth=2.5cm,hmarginratio=1:1,top=32mm,columnsep=20pt]{geometry} 
\usepackage{multicol} 

\usepackage[hang, small,labelfont=bf,up,textfont=it,up]{caption} 
\usepackage{booktabs} 
\usepackage{float} 

\usepackage{lettrine} 
\usepackage{paralist} 

\newtheorem{proposition}{Proposition}[section]
\newtheorem{theorem}[proposition]{Theorem}
\newtheorem{lemma}[proposition]{Lemma}
\newtheorem{definition}[proposition]{Definition}
\newtheorem{corollary}[proposition]{Corollary}

\theoremstyle{remark}
\newtheorem{remark}[proposition]{Remark}
\newtheorem{example}[proposition]{Example}

\makeatletter

\@addtoreset{equation}{section}
\makeatother

\def\CC{\mathcal{C}}
\def\kk{\mathbf{k}}  

\def\CM{\mathcal{M}}

\def\CR{\mathcal{R}}
\def\CT{\mathcal{T}}
\def\CV{\mathcal{V}}
\def\CW{\mathcal{W}}
\def\CB{\mathcal{B}}

\def\id{{\mathrm{Id}}}

\newcommand{\C}{\mathbb C}

\newcommand{\glie}{\mathfrak{g}}
\newcommand{\ggot}{\mathfrak{g}}
\newcommand{\hlie}{\mathfrak{h}}

\newcommand{\adiag}{A_{\mathrm{diag}}}
\newcommand{\asub}{A_{\mathrm{sub}}}
\newcommand{\gdiag}{ \glie_{\mathrm{diag}}}
\newcommand{\hdiag}{ \hlie_{\mathrm{diag}}}
\newcommand{\gsub}{ \glie_{\mathrm{sub}}}
\newcommand{\hsub}{ \hlie_{\mathrm{sub}}}
\newcommand{\glsub}{ \mathfrak{gl}_{\mathrm{sub}}}

\newcommand{\liealg}{\mathrm{Lie}}

\newcommand\cercle[1]{{\large \textcircled{\normalsize{#1}}}}
\newcommand\displaystylee{ }

\begin{document}
\sloppy
\title[Galois-Lie algebra of a reducible differential system]{Computing the Lie algebra of the differential Galois group: the reducible case}

\author{Thomas Dreyfus}
\address{Institut de Recherche Math\'ematique Avanc\'ee, \textsc{u.m.r.} 7501,\\
Universit\'e de Strasbourg et \textsc{c.n.r.s.},  7 rue Ren\'e Descartes, 67084
Strasbourg, France}
\email{dreyfus@math.unistra.fr}
\thanks{This work has received funding from the European Research Council (ERC) under the European Union's Horizon 2020 research and innovation programme under the Grant Agreement No 648132; it has also been partially supported by the LabEx PERSYVAL-Lab (ANR-11-LABX-0025-01) funded by the French program Investissement d'avenir, and ANR project \emph{De Rerum Natura} ANR-19-CE40-0018.
}
\author{Jacques-Arthur Weil}
\address{XLIM,  \textsc{u.m.r.} 7252, Universit\'e de Limoges et \textsc{c.n.r.s.} \\ 123 avenue Albert Thomas, 87060 Limoges Cedex, France}
\email{weil@unilim.fr}

\subjclass[2010]{}
  \renewcommand{\subjclassname}{%
    \textup{2010} Mathematics Subject Classification}


\date{\today}

\begin{abstract}
In this paper, we explain how to compute the Lie algebra of the differential Galois group of a reducible linear differential system. We achieve this by showing how to transform a  block-triangular linear differential system into a Kolchin-Kovacic reduced form. We combine this with other reduction results to propose a general  algorithm for computing a reduced form of a general linear differential system. In particular, this provides directly the Lie algebra of the differential Galois group without an a priori computation of this Galois group. 
\end{abstract} 

\maketitle 

\setcounter{tocdepth}{1}
\tableofcontents


\section*{Introduction}
Let $\mathcal{A}(x)\in \CM_n(\mathbf{k})$ denote an $n\times n$ matrix with coefficients in a differential field
$(\mathbf{k},\partial)$ of characteristic zero, {for instance $\mathbf{k}=\widebar{\mathbb{Q}}(x)$}.
We consider the linear differential system $[\mathcal{A}] : \;  Y'(x)=\mathcal{A}(x)Y(x)$.  
The \emph{differential Galois group} $G$ of $[\mathcal{A}]$ is an algebraic group which somehow measures the algebraic relations among the entries of a fundamental solution matrix of $[\mathcal{A}]$.
The aim of this paper is to explain how to compute effectively the Lie algebra $\mathfrak{g}$ of the differential Galois group $G$, without computing $G$.\\ \par 
\noindent\textbf{Goal of the paper.}
Given an invertible matrix $P(x)\in \mathrm{GL}_n(\mathbf{k})$,  the change of variable (``gauge transformation'') ${Y(x)=P(x).Z(x)}$ produces the linear differential system noted $Z'(x)=P(x)[\mathcal{A}(x)] \; Z(x)$, with ${P(x)[\mathcal{A}(x)]:=P^{-1}(x)\mathcal{A}(x)P(x)- P^{-1}(x)P'(x)}$. The differential system $[\mathcal{A}]$ is called \emph{reducible} if there exists a gauge transformation $P(x)$, such that $A(x):=P(x)[\mathcal{A}(x)]$ is of the form
$$
 A(x)=\left(\begin{array}{c|c}
	      A_1(x) & 0 \\\hline
	      S(x) & A_2(x) \\
	   \end{array}\right).
$$
There exist algorithms to test and realize this \emph{factorization}; they
may be found in    \cite{Si96a,Ba07a} for the completely reducible case, and in the appendix of \cite{CoWe04a}
(and references therein) for the general case. 
\\

Let $G$ denote the differential Galois group of $[A] : \;  Y'(x)=A(x)Y(x)$. 
Let $\mathfrak{g}$ be its Lie algebra. 
In this paper, we show how to compute $\mathfrak{g}$ by using the theory of reduced forms of linear differential systems.	  
Finding a \emph{reduced form} of $[A]$ amounts to finding a gauge transformation $P(x)$ 
(possibly over an algebraic extension $\mathbf{k_0}$ of $\mathbf{k}$) such that $P(x)[A(x)] \in \mathfrak{g}(\mathbf{k_0})$. 
This is similar to the Lie-Vessiot-Guldberg theories of reduction of connections in differential geometry 
(see \cite{BlMo10a,BlMo12a} for the latter and their connections with the Kolchin-Kovacic theory of reduced forms). 
Our contribution is to provide an algorithm to compute  such a reduction matrix $P(x)$ for a reducible system.
\par 
In \cite{BaClDiWe16a}, it is explained how to put a completely reducible block-diagonal system into reduced form. 
We will show that, to reduce $[\mathcal{A}]$, it is thus sufficient to be able to reduce $[A]$ under the assumption that the block diagonal differential system
$$ Y'(x) = A_{\mathrm{diag}}(x) Y(x), 
	   \quad \textrm{ with } \quad 
	A_{\mathrm{diag}}(x) = \left(\begin{array}{c|c} 
	      A_1(x) & 0 \\\hline
	      0 & A_2(x) \\
	   \end{array}\right), $$
is in reduced form. 
We, with A. Aparicio-Monforte, had solved this problem in  \cite{ApDrWe16a} in the special case when the Lie algebra of $A_{\mathrm{diag}}(x)$ is abelian. This was extended by Casale and the second author in \cite{CaWe15a} to families of $\mathrm{SL}_2$-systems.
In this paper, we treat the problem in the general case.  A review of this work with an emphasis on down-to-earth exposition, relations to questions of theoretical physics and examples can be found in \cite{dreyfus2021differential}.
\\ \par

\noindent\textbf{
General algorithms for computing differential Galois groups.} 
Using the classification of the algebraic subgroups of $\mathrm{SL}_{2}$, Kovacic gave an efficient algorithm for  computing liouvillian solutions, which in turn allows to essentially obtain the differential Galois group when $n=2$. This approach was systematized by Singer and Ulmer in \cite{SiUl93a,SiUl93b} and then \cite{SiUl97a}, notably  in the case $n=3$. Let us now describe general procedures that work for an arbitrary $n$. 
Compoint and Singer gave a decision procedure in \cite{CoSi99a} to compute the differential Galois group in the case of completely reducible (direct sums of irreducible) systems.
Berman and Singer gave an algorithm extending \cite{CoSi99a} for a large class of reducible systems  \cite{Be02a,BeSi99a}.
Using model theory, Hrushovski gave in \cite{Hr02a} the first general decision procedure computing the Galois group.  It  was  clarified and improved by Feng in \cite{Fe15a}, see also \cite{sun2018new}.  More recently, the paper  \cite{amzallag2018degree} introduces new ideas to further improve the bounds in Hrushovski's algorithm. 
 A symbolic-numeric algorithm was proposed by van der Hoeven in \cite{Ho07b}, based on the Schlesinger-Ramis density theorems. 
None of these general algorithms is currently implemented, either because their complexity is prohibitive (especially Hrushovski's algorithm) or because it is not yet known how to implement some of the required building blocks.
\\ \par

\noindent\textbf{General algorithms for computing reduced forms.} 
In the last decade, a  strategy has been developed to compute the Lie algebra, instead of the Galois group, by computing
a reduced form of the differential system. 
The Kolchin-Kovacic reduction theorems appear
in the Kovacic program on the inverse problem \cite{Ko69a,Ko71a} and in works of Kolchin on the logarithmic derivative \cite{Ko73a,Kol99a}.
Further studies of Lie-Kolchin reduction methods are carried out by D. Blazquez and J.-J. Morales in \cite{BlMo10a,BlMo12a}.
As a computation strategy, reduced forms are used in \cite{ApWe11a,ApWe12b,ApDrWe16a} (the strategy in \cite{NgPu10a} is also related to this approach). A characterization of reduced forms in terms of invariants is proposed in \cite{ApCoWe13a}; the latter paper also contains a decision procedure for putting the system into reduced form when the Galois group is reductive. A more elaborate, and much more efficient, algorithm is given in \cite{BaClDiWe16a} in the case of an absolutely irreducible system.
\\ \par

\noindent\textbf{Some motivations for reductions of reducible systems.}
In  papers on differential Galois theory, the case of a reducible system is sometimes brushed aside for two reasons. First, if one solves the irreducible diagonal blocks, then the full system can be solved by variation of constants. Second, a generic system is irreducible anyway so it may seem futile, at first glance, to spend energy on rare reducible systems.

Regarding the first objection, it would require to first solve irreducible systems,   and then construct a big Picard-Vessiot extension; variation of constants would then require the computation of integrals of transcendental functions. 
Namely, in the above notations, a fundamental matrix is 
\[ U =\left(	\begin{array}{c|c} 
			U_{1} & 0 \\\hline U_2 V  & U_{2}
		\end{array}\right), 
\quad \textrm{ with } \quad
 \left\{\begin{array}{ccl}
	U_i'  &=& A_i U_i , \\
	V' &=&  U_2^{-1} S U_1,
  \end{array}\right.
\]
where $S$ denotes the lower triangular block in $A$.
In contrast, the approach developed here uses essentially rational solutions of linear differential systems with coefficients in the base field; in return, it may actually be used to study properties of integrals of holonomic transcendental functions, see \cite{Be01a}.
Indeed, a reduced form gives us all algebraic relations between these integrals of holonomic functions; in particular, we will obtain a basis of transcendental integrals to express all the other ones.
\\

Regarding the second objection, it turns out that,  in many practical applications,
the differential systems or operators that occur happen to be reducible. Indeed we
next describe several examples of this.

The context which was our initial motivation is the Morales-Ramis-Sim\'o theory: it studies integrability properties of dynamical systems by studying successive differential systems, the variational equations, which can be viewed as a cascade of reducible systems. Algorithms to obtain reduced forms, and hence integrability criteria, for such systems are elaborated in \cite{ApWe11a,ApWe12b,ApDrWe16a}. 
For more general (non-integrable) non-linear differential systems, the Lie algebras of the differential Galois groups of variational equations give information on the Malgrange groupo\"{\i}d of the system. This is shown by Casale in  \cite{Ca09a} 
and developed in \cite{CaWe15a} to compute the Malgrange groupo\"{\i}d of (non-linear) second order differential equations. Once again, the fact that the variational equations are reducible systems turns out to be an important ingredient.

Reducible operators also appear very naturally in the holonomic world of  statistical mechanics or combinatorics, see
\cite{BoBoHaMaWeZe09a,BoBoHaHoMaWe11a} or the reference book \cite{Mc10a}. In this context, objects (or generating series) appear as convergent holonomic power series with integer coefficients; they are solutions of linear differential operators and their minimal operator is often reducible, see e.g. \cite{BoBoHaMaWeZe09a,BoBoHaHoMaWe11a}.

Last, we may also mention prolongations of systems which appear in works  on generic Galois groups (situations with mixed differential and $q$-difference structures), see \cite{DiHa10a} and references therein. These are also (structured) reducible linear differential systems and tools from this work may hence be used for a better understanding of generic or particular parametrized differential Galois groups. Similar prolongations also appear when studying singularly perturbed linear differential systems and studying solutions as series in the perturbation parameter, see e.g. the PhD of S. Maddah \cite{MadT} and references therein. The methods that we elaborate here may lead to simplification methods for such systems.\\ \par

\noindent\textbf{Structure of the paper. }  
The paper is organized as follows. In $\S \ref{secgalois}$, we recall some basic facts of differential Galois theory. We present the theory of reduced forms, notably the Kolchin Kovacic reduction theorem, which is the heart of our paper. In $\S \ref{sec:dec}$ we prove that the reduction matrix may be chosen to have a particular shape: it is a unipotent triangular matrix.
The action of such a gauge transformation on the matrix $A(x)$ of the system is governed by the adjoint action of the block-diagonal part of $A(x)$ on its off-diagonal parts.
The results of this first part of $\S \ref{sec:dec}$ are generalizations of \cite{ApDrWe16a}. Then, we recall  the construction of an isotypical flag, which will be adapted to the adjoint action in the reduction process. In $\S \ref{sec:exred}$ we give examples of the reduction process of $\S \ref{sec4}$.
We have chosen to take examples in increasing degrees of complexity in order to show step by step what the difficulties are. 
In
$\S \ref{sec4}$, we present the main contribution of the paper.
We explain how to put a block-triangular linear differential system into reduced form. 
Applying linear algebra and standard module-theoretic tools (isotypical decomposition, flags of indecomposable modules, etc.),
we generalize the techniques of \cite{ApDrWe16a} to this \emph{non-abelian} setting\footnote{The main difference between this  paper and \cite{ApDrWe16a} is that, in the previous paper, the Lie algebra of $A_{\mathrm{diag}}(x)$ was abelian. This had the consequence that the eigenvalues of the adjoint action belonged to $\mathbf{k}$ and a convenient Lie subalgebra of $A(x)$ admitted a basis of constant matrices in which the matrix associated to the adjoint action was in Jordan normal form. The reduction problem was then reduced to rational solutions of first order scalar linear differential equations. This is no longer true here.}.
The gauge transformation which reduces the system is then derived from the computation of rational solutions of successive linear differential systems with parametrized right-hand-side, see Theorem \ref{theoreme-reduction}. 
We believe that this part will generally be algorithmically efficient because it uses mostly linear algebra and rational solutions of linear differential systems of bounded size. We show this in several examples; see also the maple worksheet \cite{DrWe20a}.
In $\S \ref{sec:algo}$, we 
{present another contribution.} We 
show how the results of $\S \ref{sec4}$  may be combined with other results in order to have a general algorithm for reducing a general linear differential system. 
\par
The last two short sections are mostly expository and included for self-containedness.
In $\S \ref{sec:lie}$, we explain,  given a  system in reduced form, how to compute the Lie algebra $\glie$ of the differential Galois group. In  $\S \ref{sec7}$, we describe how, having computed 
the Galois-Lie algebra $\glie$ {of a reduced linear differential system}, one can recover its differential Galois group $G$ (using connectedness). The material in  $\S \ref{sec:lie}$ and  $\S \ref{sec7}$ is mostly known.\\ \par 

\noindent\textbf{Acknowledgments. }  
	We would like to thank G. Casale, R. Feng, and M.-F. Singer 
as well as M.-A Barkatou, T. Cluzeau and L. Di Vizio
for excellent conversations regarding the material presented here. 
We specially thank both referees for many clarifying comments and suggestions.

\pagebreak[3]
\section{Differential Galois Theory and Reduced Forms}\label{secgalois}

\subsection{The Base Field} \label{base-field}
Let us consider a differential field of characteristic zero $(\mathbf{k},\partial)$, i.e. a field equipped with a derivation. We will use the classical notation $c'$, for the derivative of $c\in \mathbf{k}$. We assume that 
its constant field $\CC:=\{c\in \mathbf{k}\mid c'=0\}$ is algebraically closed. 
We need to make assumptions about our base field $\mathbf{k}$ to elaborate our algorithms. 
\begin{itemize}
\item[\cercle{1}]
First we assume that $\mathbf{k}$ is an effective field,
i.e. that one can compute representatives of the four operations $+,-,\times,/$ and one can effectively test whether two elements of $\mathbf{k}$ are equal.  
\item[\cercle{2}]
We also assume that, given a homogeneous linear differential system ${[A]:\,  Y'(x)=A(x)Y(x)}$ with $A(x)\in \CM_n(\mathbf{k})$, we can effectively find a 
basis of its rational solutions, i.e. its solutions $Y(x)\in \mathbf{k}^n$.
\item[\cercle{3}]
Finally, we assume that, given a homogeneous linear differential system ${[A]:\,  Y'(x)=A(x)Y(x)}$ with $A(x)\in \CM_n(\mathbf{k})$, we can effectively find a 
basis of its exponential, also called hyperexponential,  solutions (see \cite{BaClElWe12a}).
\end{itemize}

The standard example of such a field would be $\mathbf{k}=\CC(x)$ with $\CC=\overline{\mathbb{Q}}$. When $\mathbf{k}=\CC(x)$, a fast algorithm for rational solutions of linear differential systems is given in \cite{Ba99a}. A Maple package \textsc{IntegrableConnections}, based on ISOLDE \cite{Isolde}, for this task is proposed in \cite{IntegrableConnections}. Algorithms for $\cercle{2}$ and $\cercle{3}$ and generalizations appear in  \cite{BaClElWe12a} (and references therein).
\begin{remark} Assumption \cercle{3} is used only in the factorization algorithm which is a preliminary step to our reduction method. The specific algorithm proposed in this paper only uses the rational algorithms of assumption \cercle{2} and  also \cercle{1}.
\end{remark}
Singer showed, in \cite{Si91a}, Lemma 3.5 and Theorem 4.1, that if $\mathbf{k}$ is an elementary extension
of $\CC(x)$ or if $\mathbf{k}$  is an algebraic extension of a purely transcendental Liouvillian extension
of $\CC(x)$, then $\mathbf{k}$ satisfies the above conditions and hence suits our purposes. 

To simplify the exposition, we will further assume that $\mathbf{k}$ is a $\mathcal{C}^{1}$-field\footnote{A field $\mathbf{k}$ is a $\mathcal{C}^{1}$-field when every non-constant homogeneous polynomial $P$ over $\mathbf{k}$ has a non-trivial zero provided that the number of its variables is more than its degree. For example, $\CC(x)$ is a $\mathcal{C}^{1}$-field and any algebraic extension of a $\mathcal{C}^{1}$-field is a $\mathcal{C}^{1}$-field (Tsen's theorem).}.

\subsection{Differential Galois Theory}\label{sec12}
We review classical elements of differential Galois theory. We refer to \cite{PS03} or \cite{CrHa11a,Si09a} for details and proofs.
 Let us consider a linear differential system of the form $[A]:\,  Y'(x)=A(x)Y(x)$, with $A(x)\in \CM_n(\mathbf{k})$. 
 A \emph{Picard-Vessiot extension} for $[A]$ is a differential field extension $K$ of $\mathbf{k}$, generated over $\mathbf{k}$ by the entries of a fundamental solution matrix of $[A]$ and such that the field of constants of $K$ is $\CC$. 
 The Picard-Vessiot extension $K$ exists and is unique up to differential field isomorphism. 
  \\ \par 
 The \emph{differential Galois group} $G$ of the system $[A]$ is the group of field automorphisms of the Picard-Vessiot extension $K$ which commute with the derivation and leave all elements of $\mathbf{k}$ invariant. 
Let $U(x)\in \mathrm{GL}_n(K)$ be a fundamental solution matrix of $ Y'(x)=A(x)Y(x)$ with coefficients in $K$. 
 For any $ \varphi\in G$, $\varphi(U(x))$ is also a fundamental solution matrix,
 so there exists a constant matrix
 $C_{\varphi} \in  \mathrm{GL}_n\left(\CC\right)$ such that $\varphi(U(x)) = U(x).C_{\varphi}$.
 The map $\rho_{U}:  \varphi \longmapsto C_{\varphi}$ 
is an injective group morphism. 
The group $G$, identified with $\hbox{Im } \rho_{U}$, may be viewed as a linear algebraic subgroup of~$\mathrm{GL}_n\left(\CC\right)$. 
\\

The \emph{Lie algebra} $\mathfrak{g}$ of the linear algebraic group $G\subset \mathrm{GL}_n\left(\CC\right)$
is the tangent space to $G$ at the identity. Equivalently, it is the set of matrices $N\in \CM_n(\CC)$ such that $\id_n + \varepsilon N$ 
satisfies the defining equations of the algebraic group $G$ modulo $\varepsilon^2$.
The Lie algebra $\mathfrak{g}$ of the differential Galois group is referred to as the \emph{Galois-Lie} algebra of the differential system $[A]$. 
The dimension of the Lie algebra $\mathfrak{g}$, as a vector space, is the transcendence degree of a Picard-Vessiot extension. 
Consequently, if we are able to compute the dimension of the Galois-Lie algebra, it will help us to prove results of algebraic independence among solutions of a linear differential system. 
The following example illustrates this by showing how our techniques allow to prove or disprove algebraic dependence of integrals of $D$-finite functions.

\begin{example}
Let $A_1:= \left( \begin {array}{cc} 0&1\\  {\frac {3\,{x}^{2}-
6\,x+7}{144\,x \left( x-1 \right) ^{2}}}&-\frac{2}{3x}-\frac{2}{3\left( x
-1 \right)}\end {array} \right)
.$
A basis of solutions of the equation associated to $[A_1]$ is given by Heun functions $f_1(x),f_2(x)$. The Kovacic algorithm, see \cite{Kov86,HoWe05a}, shows that the differential Galois group is a finite extension of $\mathrm{SL}_{2}(\CC)$; so the Galois Lie algebra has dimension  $3$. 
The system $[A]$ given by
\[ 
A=  \left( \begin {array}{ccc} 0&1&0\\  {\frac {3\,{x}^{
2}-6\,x+7}{144\, \left( x-1 \right) ^{3}{x}^{2}}}&-\frac{2}{3x}-\frac{2}{3\left( x
-1 \right)}&0\\  1&0&0\end {array}
 \right)   
\]
has fundamental  solution matrix
\[
 \left(\begin{matrix} 
       f_1(x) & f_2(x) & 0 \\
        f_1'(x) & f_2'(x) & 0 \\
        \int^x f_1(t) {\rm dt} & \int^x f_2(t) {\rm dt} & 1
    \end{matrix} \right).
  \]
 One can show, for example with the techniques of this paper, that the Galois-Lie algebra
of $[A]$ has dimension $5$. It follows that the $\int f_i(t) dt$ are transcendental
and algebraically independent over $\CC(x)(f_1,f_2,f_1',f_2')$.
\\
However, suppose we had started from 
\[ A_1 = \left( \begin {array}{cc} 0&1\\  \frac{1}{36}\,{\frac {1}{x
 \left( x-1 \right) }}&-{\frac {7}{12\,x}}-\frac{1}{6\left( x-1 \right)}\,  \end {array} \right).
\]
Its differential Galois group is also a finite extension of the group $\mathrm{SL}_{2}(\CC)$. The maple implementation of 
\cite{HoWe05a} gives us two hypergeometric 
solutions ${f_1(x) = {}_2F{}_1([-1/3, 1/12], [7/12])(x)}$ and $f_2(x) = x^{5/12}{ }_2F{}_1([1/12, 1/2], [17/12])(x).$
The Galois-Lie algebra of $[A]$, { with $A=  \left( \begin {array}{ccc} 0&1&0\\  \frac{1}{36}\,{\frac {1}{x
 \left( x-1 \right) }}&-{\frac {7}{12\,x}}-\frac{1}{6\left( x-1 \right)}&0\\  1&0&0\end {array}
 \right) $,} turns out to have dimension $3$ and 
reduction techniques applied to $[A]$ give us the relations 
\[
\int^x \!f_i \left( t \right) \,{\rm dt}=
	- \frac {9}{11}\,x \left( x-1 \right) f_i'\left( x \right) 
	+ \frac{15}{44} \left( 3 \, x -1 \right)  f_i \left( x \right) 
	+\frac {9}{11}\,{c_i }
\]
satisfied by the $f_i$, for some constants $c_i$.
\end{example}
For a factorized reducible system
 $[A]$ of the form \[A(x) = \left(\begin{array}{c|c} A_{1}(x) & 0 \\\hline S(x) & A_{2}(x)\end{array}\right)
 = A_{diag}(x) + A_{sub}(x), \]
we have a fundamental solution matrix of the form
\[ U =\left(\begin{array}{c|c} U_{1} & 0 \\\hline U_2 V & U_{2}\end{array}\right) 
	= \left(\begin{array}{c|c} U_{1} & 0 \\\hline 0 & U_{2}\end{array}\right) \left(\begin{array}{c|c} \mathrm{Id}_{\mathrm{n_{1}}} & 0 \\\hline  V & \mathrm{Id}_{\mathrm{n_{2}}}\end{array}\right). 
\]
Once $U_1$ and $U_2$ are known, $V$ is given by integrals : $V'=U_2^{-1} S U_1$. 
Let $K_{diag}:=\mathbf{k}(U_1,U_2)$ be a Picard-Vessiot extension of $\mathbf{k}$ for $[A_{diag}]$, with differential Galois group $G_{diag}$.
Then $K:=K_{diag}(V)$ is a Picard-Vessiot extension of $\mathbf{k}$ for $[A]$ with differential Galois group $G$. Note that $K_{diag}$ has field of constants $\mathcal{C}$, so that we may consider the differential Galois group over $K_{diag}$. Letting $G_{u}:=Gal(K/K_{diag})$ be the differential Galois group of $[A]$
 over $K_{diag}$, we have by Galois correspondence, see Proposition 1.34 of \cite{PS03}, that $G_u$ is the set of elements of $G$ of the form 
$\left(\begin{array}{c|c} \mathrm{Id}_{\mathrm{n_{1}}} & 0 \\\hline G_{2,1} & \mathrm{Id}_{\mathrm{n_{2}}}\end{array}\right)$ as the $U_i$ are fixed. 
 Then,  $G_u\triangleleft G$ and
 \[  \overbrace{\mathbf{k} \hspace{0.2cm}\subset\hspace{0.2cm} K_{diag}}^{G_{diag} \simeq G/G_{u}}\hspace{-0.0cm}:= 
	\overbrace{\mathbf{k}(U_1,U_2) \hspace{0.2cm}\subset\hspace{0.2cm} K}^{G_{u}\triangleleft G}:=K_{diag}(V).
\]
\begin{remark}\label{quotient}
Given $g\in G$, we have $g(U_i)=U_i.C_i$ for invertible constant matrices $C_i$.  
Then $\left. G_{diag} \simeq \left\{\left(\begin{array}{c|c} C_{1} & 0 \\\hline 0 & C_2\end{array}\right) 
	\; \right| \; \exists g\in G, g(U_i)=U_i.C_i \textrm{ for } i=1,2 \right\}. $
In other words, for any matrix $\left(\begin{array}{c|c} C_{1} & 0 \\\hline 0 & C_2\end{array}\right) \in G_{diag}$, there exists a matrix $M\in G$ with $M=\left(\begin{array}{c|c} C_{1} & 0 \\\hline C_{2,1} & C_2\end{array}\right)$.
\end{remark}
We define $C_g:=C_2 g(V)-V C_1$.  In virtue of   
$$g(V')=g(U_2^{-1} S U_1)=C_2^{-1} U_2^{-1} S U_1 C_1=C_2^{-1} V' C_1,$$
we find that $C_g$  is a constant matrix. Since  $ g(V)=C_2^{-1}V C_1+C_2^{-1} C_g$,
 we see that  
$g(U_2V) = U_2 V C_1 + U_2 C_g $ and 
\[ g(U) = \left(\begin{array}{c|c} U_{1} & 0 \\\hline U_2 V & U_{2}\end{array}\right) \cdot
	\left(\begin{array}{c|c} C_{1} & 0 \\\hline C_g & C_2\end{array}\right)  
	=  U\cdot 
	\left(\begin{array}{c|c} C_{1} & 0 \\\hline 0 & C_2\end{array}\right)
	\left(\begin{array}{c|c} \mathrm{Id}_{\mathrm{n_{1}}} & 0 \\\hline C_2^{-1}\cdot C_g & \mathrm{Id}_{\mathrm{n_{2}}} \end{array}\right).
\]
 {Last, we recall a useful lemma to switch from group to Lie algebra in specific cases:}
\begin{lemma}\label{lem7}
Let $\mathfrak{n}\subset \CM_n\left(\CC\right)$ be a $\CC$-vector space of 
lower triangular matrices with zero entries on the diagonal.
Assume that, for all $N,N'\in \mathfrak{n}$, $N\cdot N'=(0)$. 
Then,  ${U:=\Big\{\id_n+N, N\in  \mathfrak{n}\Big\}}$ is a connected algebraic group
and $\mathfrak{n}$ is its Lie algebra.\par 
Furthermore, we have two bijective maps which are inverses of each other
$$\begin{array}{cccc}
\exp : &  \mathfrak{n} & \longrightarrow & U \\
& N&\mapsto  &\id_n+N \\
 \log : &  U & \longrightarrow & \mathfrak{n} \\
 &\id_n+N&\mapsto  &N.
\end{array}$$
\end{lemma}

\begin{proof}
The algebraic group $U$ is abelian thanks to the assumption; the fact that $\mathfrak{n}$ is
its Lie algebra is easily derived from the definition.
Let  $N\in \mathfrak{n}$. 
We have $\exp(N) = \id_n + N$ because, by assumption, $N^2=0$.
The same argument shows that ${\log(\id_n + N)=N}$. 
It follows that $\exp$ and $\log$ are bijective on the required sets and 
are inverses of each other.  This also proves the connectedness of $U$.
\end{proof}

\subsection{Reduced Forms of Linear Differential Systems} \label{section-reduced}
Let  $A(x)\in \CM_n\left(\mathbf{k}\right)$, $G$ be the differential Galois group of $[A]: \,  Y'(x)=A(x)Y(x)$
and $\mathfrak{g}$ its Lie algebra. As usual, the notation $\mathfrak{g}(\mathbf{k})$ stands for the extension of scalars
$\mathfrak{g}(\mathbf{k})=\mathfrak{g}\otimes_{\CC} \mathbf{k}$. Let $\overline{\mathbf{k}}$ be the algebraic closure of $\mathbf{k}$.

\begin{definition}\label{gauge-def}
Let us consider ${A(x),B(x)\in \CM_n(\mathbf{k})}$. 
The two linear differential systems~${[A]:\, Y'(x)=A(x)Y(x)}$ and~$[B]:\, Z'(x)=B(x)Z(x)$ are called \emph{equivalent over~$\mathbf{k}$} (or \emph{gauge equivalent} over~$\mathbf{k}$) 
when there exists~$P(x)\in \mathrm{GL}_n(\mathbf{k})$
such that 
\[ B(x)=P^{-1}(x)A(x)P(x) - P^{-1}(x)P'(x).\]
The notation is $B=P[A]$ and $P$ is called a \emph{gauge transformation matrix}.
\end{definition}
Solutions of $[A]$ and $[B]$ are then linked by the relation $Y(x)=P(x)Z(x)$.

\begin{definition}
{Let ${A(x)\in \CM_n(\mathbf{k})}$. }
We say that the system $[A]: \, Y'(x)=A(x)Y(x)$ is \emph{in reduced form} (or in \emph{Kolchin-Kovacic reduced form})
{when} $A(x) \in \mathfrak{g}(\mathbf{k})$. 
\par
Otherwise, we say that a matrix 
	{ $B(x)\in  \CM_n\left(\overline{\mathbf{k}}\right)$ (resp. a system $[B]$) is a  \emph{reduced form of $[A]$}  when there exists 
	$P(x) \in \mathrm{GL}_n(\overline{\mathbf{k}})$ }
	such that $B(x)=P(x)[A(x)]$ and $B(x)$ is in reduced form, i.e.
	$B(x) \in \mathfrak{g}(\overline{\mathbf{k}})$.
\end{definition}

The existence and relevance of reduced forms are given by the following  Kolchin-Kovacic reduction result.
A proof can be found in \cite{PS03}, Proposition 1.31 and Corollary 1.32. See also \cite{BlMo10a}, Theorem~5.8, and \cite{ApCoWe13a}, {\S}~5.3 after Remark~31.

\begin{proposition}[Kolchin-Kovacic reduction theorem]\label{propo1}
Let $A(x)\in \CM_n(\mathbf{k})$. 
Let $G$ be the differential Galois group of the differential system $[A]:\, Y'(x)=A(x)Y(x)$ and $\mathfrak{g}$ be the Lie algebra of $G$.
Let $H\subset \mathrm{GL}_n\left(\CC\right)$ be a connected linear algebraic group, with Lie algebra $\mathfrak{h}$, such that 
	$A(x) \in \mathfrak{h}(\mathbf{k})$.
\begin{enumerate}
\item The Galois group $G$ is contained in (a conjugate of) $H$.
\item There exists a gauge transformation $P(x)\in H(\overline{\mathbf{k}})$ such that $P(x)[A(x)]\in \mathfrak{g}(\overline{\mathbf{k}})$. \\
If we further assume that $G$ is connected and that 
$\mathbf{k}$ is a $\mathcal{C}^{1}$-field,
then there exists a gauge transformation $P(x)\in H(\mathbf{k})$ such that $P(x)[A(x)]\in \mathfrak{g}(\mathbf{k})$.
\end{enumerate}
\end{proposition}

We now construct a Lie algebra $\mathfrak{h}$ such that $\mathfrak{h}$ is the Lie algebra of some algebraic group $H$, 
	$A(x) \in \mathfrak{h}(\mathbf{k})$ and $\mathfrak{h}$ has minimal dimension for that property.

Following \cite{WeNo63a,ApCoWe13a}, a  \emph{Wei-Norman decomposition} of $A(x)$ is a finite sum of the form
$$A(x)=\sum a_{i}(x)M_{i},$$
where the
matrices $M_{i}$ have
 coefficients in $\CC$ and the $a_{i}(x)\in \mathbf{k}$ form a basis of the $\CC$-vector space spanned by the entries of $A(x)$. 
The $M_{i}$ depend on the choice of $a_{i}(x)$ but the $\CC$-vector space generated by the $M_{i}$ is independent of the choice of the $a_{i}(x)$. This shows that the 
notation $\liealg(A)$ below does not depend upon the choice of the Wei-Norman decomposition and is well defined.

Recall that a Lie algebra is called an \emph{algebraic Lie algebra} when it is the Lie algebra of an algebraic group.
\begin{definition}  
For a matrix $A(x)\in \CM_n(\mathbf{k})$, 
let constant matrices $M_{i}$ denote the generators of a Wei-Norman decomposition of $A(x)$. 
\\
We define $\liealg(A)$, 
called \emph{the Lie algebra associated to $A$},  
as the smallest algebraic Lie algebra which contains all matrices $M_i$, 
i.e. the algebraic envelope of the Lie algebra generated by the $M_i$.
\end{definition}

The link between $\glie$ and $\liealg (A)$ is made in the following remark.
\begin{remark}\label{rem3}
{By Proposition \ref{propo1}, $\mathfrak{g}\subset \liealg (A)$. We see that the system $Y'(x)=A(x)Y(x)$ is in reduced form if and only if $\mathfrak{g}=\liealg(A)$.}
\end{remark}

The approach that we elaborate in this paper was initiated in \cite{ApWe11a,ApDrWe16a} and \cite{CaWe15a}. It is based on a criterion for reduced forms,  which is given in the following lemma.
  
  \begin{lemma}[\cite{ApDrWe16a}, Lemma 1.3] \label{reduite-minimale}
Given $A(x)\in \CM_n\left(\mathbf{k}\right)$, let $G$ be the differential Galois group of the system $[A]:\,  Y'(x)=A(x)Y(x)$
and $\mathfrak{g}$ be its Lie algebra. Let $H$ be the connected linear algebraic group whose Lie algebra is $\liealg(A)$.
Assume that $G$ is connected. 
\\
The system $[A]$ is in reduced form, i.e. $G=H$ and $\mathfrak{g}=\liealg(A)$,
if and only if, for all gauge transformation matrices $P(x)$ in $H(\mathbf{k})$, 
we have $\liealg(A) = \liealg(P[A])$.
  \end{lemma}

\section{Decomposition and Flags for the Off-Diagonal Part}\label{sec:dec}

Let us consider a matrix 
$$ A(x) := \left(\begin{array}{c|c} A_{1}(x) & 0 \\\hline S(x) & A_{2}(x) \\\end{array}\right)\in \CM_n(\mathbf{k})$$
where $A_i(x)$ are square matrices in $\CM_{n_{i}}(\mathbf{k})$.
We have $A(x)= \adiag (x)+ \asub (x)$,  where  
	$$\adiag (x):=\left(\begin{array}{c|c} A_{1}(x)  & 0 \\\hline 0 & A_{2}(x) \end{array}\right)
		\quad \mathrm{ and } \quad  
	{\asub (x):=\left(\begin{array}{c|c} 0 & 0 \\\hline S(x)& 0\end{array}\right).}$$

Let us assume that $ Y'(x)=\adiag(x)Y(x)$ is in reduced form.
The aim of the paper is to show how to then put the full system $Y' (x)=A(x)Y(x)$ in reduced form, see $\S \ref{sec4}$. We are going to see in $\S\ref{sec:algo}$ that solving  this problem will give us a complete algorithm to put a general system into reduced form.

\subsection{The {Off-Diagonal} Algebra $\glsub$}\label{sec31}

Let $\gdiag:= \liealg(\adiag)$ be the Lie algebra associated to $\adiag(x)$. 
Let $\glsub:=\left\{ \left(\begin{array}{c|c} 0 & 0 \\\hline C  & 0\end{array}\right), C\in 
 \CM_{n_2\times n_1}(\CC)\right\} $
denote the space of off-diagonal constant matrices. 

    We now list some useful simple properties of $\glsub$.
\begin{lemma}\label{diagsub}
Let us consider a block-diagonal matrix $M = \left(\begin{array}{c|c} N_{1}  & 0 \\ \hline 0 & N_{2} \end{array}\right)\in\gdiag$ 
and off-diagonal matrices 
	$B_1 =\left(\begin{array}{c|c} 0 & 0 \\\hline C_1  & 0\end{array}\right)$ and 
	$B_2=\left(\begin{array}{c|c} 0 & 0 \\\hline C_2  & 0\end{array}\right)\in \glsub $.
\begin{enumerate}
\item 
{$B_1 . B_2=(0)$ and $\glsub$ is an abelian Lie algebra.
}
\item { $M.B_1 = \left(\begin{array}{c|c} 0 & 0 \\\hline N_{2} C_1 &0  \end{array}\right)\in \glsub$
and $B_1 .M= \left(\begin{array}{c|c} 0 & 0 \\\hline C_1 N_{1}  &0  \end{array}\right) \in \glsub$.
}
\item 
{ $[M,B_1] = \left(\begin{array}{c|c} 0 & 0 \\\hline N_{2} C_1 - C_1  N_{1}  & 0\end{array}\right) \in \glsub$.}
\end{enumerate}
\end{lemma}
\begin{proof} 
This is a simple calculation. 
\end{proof}
As a consequence of the third point we obtain the following lemma.
\begin{lemma}\label{lem10}
The Lie algebra  $\glsub$  is stable under $[\gdiag,\bullet]$.
\end{lemma}

\begin{remark} \label{rem:alg lie alg} 
Lemma \ref{lem7} applied to $\glsub$ gives that $\Big\{\id_n+B, B\in \glsub \Big\}$ is a connected algebraic group with Lie algebra $\glsub$.
More generally, given a vector subspace $W$ of $\glsub$, for every $M,N\in W$, we have $MN\in (0)$. Then $W$ is an algebraic abelian Lie algebra with additive abelian group $\{ \id_n + B, B\in W \}$.
\end{remark}

\pagebreak[3]
\subsection{The Shape of the Reduction Matrix}\label{sec32}

The aim of this subsection is to generalize Theorem 3.3 of \cite{ApDrWe16a} to our context. 
As above, we consider a system $[A]$ given by
$$A(x)=\adiag(x)+\asub(x),$$  with  
$$	\adiag (x):=\left(\begin{array}{c|c} A_{1}(x)  & 0 \\\hline 0 & A_{2}(x) \end{array}\right)
	\textrm{ and } 
	\asub (x):=\left(\begin{array}{c|c} 0  & 0 \\\hline S(x) & 0  \end{array}\right).
$$
We assume in the sequel that \emph{$[\adiag]$ is in reduced form}.

\begin{theorem}\label{theo1}
There exists a gauge transformation $$P(x)\in\Big\{\id_n+B(x), B(x)\in \glsub \left(\mathbf{k}\right)\Big\},$$ such that ${ Y'(x) = P(x)[A(x)]Y(x)}$ is in reduced form.
\end{theorem}
We first prove the following auxiliary lemma.  Let $G$ be the differential Galois group of $ Y'(x)=A(x)Y(x)$ and $\glie$ be the Lie algebra of $G$. 
{Let $H$ be the connected algebraic group with Lie algebra $\liealg (A)$}. 
\begin{lemma}\label{lem4}
The differential Galois group $G$ is connected.
\end{lemma}

\begin{proof}[Proof of Lemma \ref{lem4}]
 The elements of $G $ are of the form $\left(\begin{array}{c|c} G_{1} & 0 \\\hline G_{2,1} & G_{2}\end{array}\right)\in \mathrm{GL}_n\left(\CC\right)$. 
 Let $G_u$ be the subgroup of elements of $G$ of the form $\left(\begin{array}{c|c} \id_{\mathrm{n_{1}}} & 0 \\\hline G_{2,1} & \id_{\mathrm{n_{2}}}\end{array}\right)$. 
{As we have seen in \S \ref{sec12}, 
 $G_u$ is a normal subgroup of $G$ and $G\simeq G_u\rtimes G_{\mathrm{diag}}$, where $G_{\mathrm{diag}}$ denotes the differential Galois group of $ Y'(x)=\adiag (x)Y(x)$. }
\\
 Since the system $ Y'(x)=\adiag (x)Y(x)$ is in reduced form, we find that  $G_{\mathrm{diag}}$ is connected, see Lemma 32 in  \cite{ApCoWe13a}. Now $G_u$ is a vector group and hence it is a connected linear algebraic group as well.
\end{proof}

 \begin{proof}[Proof of Theorem \ref{theo1}.]
The differential Galois group $G$ satisfies the inclusion $G \subset H$ (because of the first point of Proposition~\ref{propo1}). Lemma~\ref{lem4} shows that $G$ is connected.
So we may use the second point of Proposition~\ref{propo1} to obtain the existence of ${Q(x):=\left(\begin{array}{c|c} D_{1}(x) & 0 \\\hline S_{Q}(x) & D_{2}(x)\end{array}\right)\in H(\mathbf{k})}$ such that the linear differential system $Q(x)[A(x)]$ is in reduced form. 
\par
Since $[\adiag]$ is in reduced form, Remark \ref{rem3} implies that  $G_{\mathrm{diag}}$, the differential Galois group of $[\adiag]$, admits $\gdiag$ as Lie algebra. By construction, $\liealg (A)$ is included in the smallest algebraic Lie algebra containing  $\gdiag \oplus \glsub$. As a consequence of Lemma \ref{diagsub}, we deduce that $\gdiag \oplus \glsub$ is a Lie algebra. It is even an algebraic Lie algebra  whose algebraic group is $G_{\mathrm{diag}}\times \Big\{\id_n+B, B\in  \glsub \Big\}$, proving that $\liealg (A)\subset \gdiag \oplus \glsub$. 
As $Q(x)\in H(\mathbf{k})$, we have 
	$\left(\begin{array}{c|c} D_{1}(x) & 0 \\\hline 0 & D_{2}(x)\end{array}\right) \in G_{\rm diag}(\mathbf{k})$. 
	Now, as $G_{diag}\simeq G/G_u$, 
	Remark \ref{quotient} shows that $G(\mathbf{k})$ contains a block-triangular matrix of the form
	$R(x):=\left(\begin{array}{c|c} D_{1}(x) & 0 \\\hline T_{2,1}(x) & D_{2}(x)\end{array}\right)$.
	Then \[R(x)^{-1} = \left(\begin{array}{c|c} D_{1}^{-1}(x) & 0 \\\hline -D_{2}^{-1}T_{2,1}D_{1}^{-1}(x) & D_{2}^{-1}(x)\end{array}\right) \in G(\mathbf{k}). \]
	By hypothesis, $Lie( Q[A] ) = \ggot$ (as $[Q[A]]$ is in reduced form). 
	Now, a gauge transformation of an element $Q[A]\in \ggot( \mathbf{k})$ by an element of $G(\mathbf{k})$ transforms $Q[A]$  into another element of $\ggot( \mathbf{k})$, see \cite[Proposition 5.1]{MiSi02a}.
Then, $[R^{-1}[Q[A]]]$ is in reduced form.	A fundamental solution of $[Q[A]]$ is given by $Q^{-1}U$, where $U$ is a fundamental solution of $[A]$. Therefore, $[R^{-1}[Q[A]]]$ has a fundamental solution $RQ^{-1}U=(QR^{-1})^{-1}U$,  and we find that ${R^{-1}[Q[A]]=(QR^{-1})[A]}$.
	It follows that $[(QR^{-1})[A]]$ is in reduced form. We have 
	\[ QR^{-1}  = \left(\begin{array}{c|c} \id_{\mathrm{n_{1}}} & 0 \\\hline (S_Q-T_{2,1})D_1^{-1} & \id_{\mathrm{n_{2}}}\end{array}\right) \in\Big\{\id_n+ \glsub \left(\mathbf{k}\right)\Big\}.  \]
Then, $QR^{-1}$ is the expected gauge transformation.
\end{proof}

The following corollary will be a key ingredient for the reduction procedure of $\S \ref{sec4}$. 

\begin{corollary}\label{coro1}
Assume that, for all gauge transformations of the form ${P(x)\in\Big\{\id_n+B(x), B(x)\in  \glsub \left(\mathbf{k}\right)\Big\}}$, we have 
$\liealg(A)\subseteq \liealg(P[A])$.
Then, the linear differential system $ Y'(x)=A(x)Y(x)$ is in reduced form.
\end{corollary}

\begin{proof}
Theorem~\ref{theo1} provides $P(x)=\id_n+B(x)$ with  ${B(x)\in \glsub \left(\mathbf{k}\right)}$ 
 such that the system $ Y'(x)=P(x)[A(x)]Y(x)$  is in reduced form. 
By assumption, we have ${\liealg(A) \subseteq \liealg(P[A])}$. But since $ Y'(x)=P(x)[A(x)]Y(x)$  is in reduced form,  we have $\liealg(P[A])=\mathfrak{g}$. This shows that $\liealg(A)\subseteq \mathfrak{g}$. 
By Remark \ref{rem3}, we had $\mathfrak{g}\subseteq \liealg(A)$ so  $\liealg(A)=\mathfrak{g}$,
which proves the result.
\end{proof}

\pagebreak[3]
\subsection{The Adjoint Action $\Psi = [\adiag (x),\bullet]$}\label{sec33} 

We refer to $\S \ref{secgalois}$ and $\S \ref{sec31}$ for the notations and definitions used in this subsection. 
The adjoint action is $ [\adiag (x),\bullet]$ for consistence of formulas.
In $\S \ref{sec32}$, we have proved  the existence of a gauge transformation matrix ${P(x)\in\Big\{\id_n+B(x), B(x)\in  \glsub \left(\mathbf{k}\right)\Big\}}$ such that the system ${ Y'(x) = P(x)[A(x)]Y(x)}$ is in reduced form. 
Let $B_1,\dots, B_\sigma \in \CM_n\left(\CC\right)$ be a basis of $\glsub$. 
The next proposition generalizes \cite[Proposition 3.6]{ApDrWe16a} in our context.

\pagebreak[3]
\begin{proposition}\label{propo2} 
Let $P(x):=\id_n + \displaystylee\sum_{i=1}^{\sigma} f_i (x)B_i$, with $f_i (x)\in \mathbf{k}$. 
The gauge transformation of [A] by $P(x)$ is
\[ P(x)[A(x)] = \adiag(x) + \underbrace{\asub(x) + \sum_{i=1}^{\sigma} f_i (x)[\adiag (x),B_i] -\sum_{i=1}^{\sigma}  f'_i (x) B_i.}_{\textrm{ off-diagonal part to be reduced}}\]
\end{proposition}

\begin{proof}
The gauge transformation of $[A]$ by $P(x)$ is given by the formula  ${P(x)[A(x)]=P(x)^{-1}A(x) P(x)-P^{-1}(x)P' (x)}$, see Definition \ref{gauge-def}. 
Computations are made simple by the fact that the product of two elements of $\glsub (\mathbf{k})$ is zero. 
We have the equalities ${P^{-1} (x)= \id_n - \displaystylee\sum_{i=1}^{\sigma} f_i (x)B_i }$ and ${P^{-1}(x)A(x) =  A(x) - \displaystylee\sum_{i=1}^{\sigma} f_{i} (x)B_i \adiag (x)}$. 
As 
$A(x)=\adiag (x)+\asub (x)$, we  find that

$$\begin{array}{lll}
P(x)^{-1}A(x) P(x) &=& \left(\adiag (x)+\asub (x) -\displaystylee\sum_{i=1}^{\sigma} f_{i}(x) B_i \adiag (x)\right)\left(\id_n + \displaystylee\sum_{i=1}^{\sigma} f_i (x)B_i\right)\\
&= &\adiag (x)+\asub (x) + \displaystylee\sum^{\sigma}_{i=1} 
\left( -f_{i}(x) B_i \adiag (x)+ f_i (x) \adiag (x)B_i \right)\\
& = & \adiag (x)+\asub (x) + \displaystylee\sum_{i=1}^{\sigma} f_i (x)[\adiag (x),B_i].
\end{array}$$
Similarly, we have
\[
{P^{-1}(x)P' (x) =\left(\id_n - \displaystylee\sum_{j=1}^{\sigma} f_j(x) B_j\right) \left( \displaystylee\sum_{i=1}^{\sigma}  f'_i (x) B_i\right)=\displaystylee\sum_{i=1}^{\sigma} f'_i (x)B_i.}
\]
This yields the desired result.
\end{proof}

The Lie algebra $\glsub$ is stable under the bracket  $[\gdiag,\bullet]$, see Lemma~\ref{lem10}.
{ 
This implies that the  $\mathbf{k}$-linear map ${\Psi:=[\adiag (x),\bullet]}$, which is the adjoint action of $\gdiag \left(\mathbf{k}\right)$ on $\glsub \left(\mathbf{k}\right)$, is well defined:
\[\begin{array}{cccc}
\Psi: & \glsub \left(\mathbf{k}\right) & \longrightarrow & \glsub \left(\mathbf{k}\right) \\ 
 & B (x)&\longmapsto  & [\adiag (x),B(x)].
\end{array}\]
Writing
\[ B(x) = \left(\begin{array}{c|c} 0& 0 \\\hline B_s(x) & 0\end{array}\right),\]
the action of $\Psi$ on $B$ induces an action $\bar\Psi$ on $B_s$ given by
	\[\bar\Psi(B_s(x))= A_2(x) B_s(x)-B_s(x) A_1(x).\]

\begin{lemma} \label{matrix-of-psi}
With the above notations, the matrix $\Psi$ 
of the adjoint action of $\adiag(x)$  on $\glsub$ is 
\[ \Psi=
A_2 \otimes \id_{\mathrm{n_{1}}} - \id_{\mathrm{n_{2}}}\otimes A_1^T.\]
\end{lemma}	

\begin{proof}
The row-\texttt{vec} operator 
transforms a matrix into a vector by transposing each row and stacking them into a vector.
It it well known that, for general matrices, the row-\texttt{vec} operator satisfies  
	$\texttt{vec}(MX)=(M\otimes \id)\texttt{vec}(X)$ and 
	$\texttt{vec}(XM)=(\id\otimes M^{T})\texttt{vec}(X)$,
see \cite{PePeot08a}, Section 10.2. 
It follows that $\texttt{vec}( \bar\Psi(B_s ) ) = 
	\left( A_2 \otimes \id_{\mathrm{n_{1}}} - \id_{\mathrm{n_{2}}}\otimes A_1^T\right)
	\cdot \texttt{vec}(B_s)$. 
\end{proof}
}
\begin{remark} \label{psi-reduit}
In this remark, we use the language of differential modules as in \cite{PS03,ApCoWe13a}.
The differential system 
{
$[A_2 \otimes \id_{\mathrm{n_{1}}} -  \id_{\mathrm{n_{2}}} \otimes A_1^T]$
corresponds to the differential module $\CM_2 \otimes \CM_1^\star$,}
where the connections on each $\CM_i$ have matrices $A_i$ respectively. 
Let us show that $[\Psi]$ is in reduced form.
Indeed, $\Psi$ is the matrix of {the connection on  $\CM_2 \otimes \CM_1^\star$}
which is a submodule of a tensor construction on $\CM_1\oplus \CM_2$.
So any semi-invariant of $\CM_2 \otimes \CM_1^\star$ in a construction, i.e.~an exponential solution of the corresponding differential system,
can be extended, by adding zeroes,  into a semi-invariant of $\CM_1\oplus \CM_2$.
Now, the matrix of the connection on 
$\CM_1\oplus \CM_2$ is $\adiag$ which is in reduced form. 
So, Theorem~1 of \cite{ApCoWe13a} shows that any semi-invariant of $\CM_1\oplus \CM_2$ has constant coefficients. It follows that any semi-invariant of $\CM_2 \otimes \CM_1^\star$ will then have constant coefficients.  So, Theorem~1 of \cite{ApCoWe13a} implies that $[\Psi]$ is in reduced form.
\end{remark}

\subsection{Decomposition of $\glsub$ into $\Psi$-spaces}\label{seciso}

\subsubsection{Isotypical decomposition of $\glsub$ into $\Psi$-spaces}
Proposition~\ref{propo2} shows that reduction will be essentially governed by the adjoint map $\Psi$.
We had the Wei-Norman decomposition $\adiag (x)=\displaystylee \sum_{i=1}^{\delta}g_{i}(x)M_{i}$, where
the $g_{i}(x)\in \mathbf{k}$ were $\CC$-linearly independent. 
For each ${i \in  \{ 1, \dots, \delta\}}$,
we define the $\CC$-linear map
$\Psi_{i}: \begin{array}{lll}\glsub&\longrightarrow& \glsub\\
B& \longmapsto&  [M_{i},B]\end{array}$
so that $\Psi=\displaystylee \sum_{i=1}^{\delta}g_{i}(x)\Psi_{i}$.

\begin{definition} \label{psi-space}
Consider a vector space $W \subset \glsub $. 
We say that $W$ is a \emph{$\Psi$-space} 
when $\Psi (W )\subset W\otimes_{\CC} \mathbf{k}$.
\\ 
Let $\CR:=\CC[\Psi_1,\ldots,\Psi_{\delta}]$ denote the $\CC$-algebra generated by the $\Psi_i$.
We say that $W$ is an \emph{$\CR$-module} (or submodule of $\glsub$)
when, for all $i \in  \{ 1, \dots, \delta\}$, 
	$\Psi_{i} (W)\subset W$.
\end{definition}

Note that, by Remark \ref{rem:alg lie alg}, any vector subspace of $\glsub$ is an algebraic Lie algebra so any $\Psi$-space is an algebraic Lie algebra.
The map $\Psi$ acts naturally on the Lie algebra $\glie$; its action on the off-diagonal matrices of $\glie$ will govern our reduction strategy, as suggested by the following lemma.
\begin{lemma}\label{lem5}
Let $\glie_{sub}:= \glsub\bigcap \glie$ denote the subspace of off-diagonal matrices of $\glie$. Then  
$\glie_{sub}$ is a $\Psi$-space. 
\end{lemma}
\begin{proof}
Let $\left(\begin{array}{c|c} C_1 & 0\\\hline 0 & C_2\end{array}\right)\in \gdiag$; as noted in Remark \ref{quotient}, there exists a corresponding  element $\left(\begin{array}{c|c} C_1 & 0 \\\hline C & C_2\end{array}\right)\in \glie$.  Let $\left(\begin{array}{c|c} 0& 0\\\hline C' & 0\end{array}\right)\in \glie_{sub}$. We have
\[
\left[\left(\begin{array}{c|c} C_1 &0 \\\hline C & C_2 \end{array}\right), \left(\begin{array}{c|c} 0& 0\\\hline C' & 0\end{array}\right)\right] =\left(\begin{array}{c|c} 0 &0 \\\hline C_2 C'-C'C_1 & 0 \end{array}\right)\in \glie_{sub}
\]
and so $\gsub$ is stable under the bracket with $\gdiag$.
Since $[\adiag]$ is in reduced form, the $M_i$ of its Wei-Norman decomposition are in $\gdiag$ so that $\gsub$ is stable under the bracket with each $M_i$.  The result follows with $\Psi=\displaystylee \sum_{i=1}^{\delta}g_{i}(x)\Psi_{i}$.
\end{proof}
In this section, we describe the structure of the $\Psi$-subspaces of $\glsub$ and how to  compute them in order to be able to reduce the off-diagonal part of the Lie algebra of the system.

\begin{lemma}\label{lem3}
A vector subspace $W$ of $\glsub$ is a $\Psi$-space if and only if for all 
$i \in \{1, \dots, \delta\}$, 
$\Psi_{i} (W)\subset W$, i.e. if and only if $W$ is an $\CR$-module.
\end{lemma}

\begin{proof}
By construction, we have $\Psi=\displaystylee \sum_{i=1}^{\delta}g_{i}(x)\Psi_{i}$. If for all $i \in \{1, \dots, \delta\}$,
${\Psi_{i} (W)\subset W}$, it is clear that $\Psi(W)\subset W\otimes_{\CC}\mathbf{k}$.
Conversely, let us assume that $W$ is a $\Psi$-space.  
Let $E_{1},\dots,E_{\kappa}$ be a basis of $W$; we complete it into a basis $E_{1},\dots,E_{\sigma}$ of  $\glsub$. We recall that the $\Psi_{i}$ are $\CC$-linear maps. 
For $i \in \{1, \dots, \delta\}$  and $u\in W$, let $c_{i,j,u}\in \CC$ such that $\Psi_{i}(u)=\displaystylee\sum_{j=1}^{\sigma}c_{i,j,u}E_{j}$. 
Since $\Psi=\displaystylee \sum_{i=1}^{\delta}g_{i}(x)\Psi_{i}$, we find that $\Psi(u)=\displaystylee\sum_{i=1}^{\delta}\sum_{j=1}^{\sigma}g_{i}(x)c_{i,j,u}E_{j}$. 
For $u\in W$, the fact that $\Psi (u)\in W{\otimes_{\CC}\mathbf{k}}$ implies that we have, for all  
$j\in\{ \kappa +1, \ldots, \sigma\}$,
$\displaystylee \sum_{i=1}^{\delta} g_{i}(x)c_{i,j,u}=0$. But the $g_{i}(x)\in \mathbf{k}$ are $\CC$-linearly independent
so, for all $i\in \{ 1,\dots, \delta\}$ and all $j\in \{\kappa+1,\dots, \sigma\}$,  we have $c_{i,j,u}=0$. This proves that, for all  
$i\in \{1,\dots, \delta\}$, $\Psi_{i} (u)\in W$
and hence $\Psi_{i} (W)\subset W$. 
\end{proof}

Lemma \ref{lem3}, applied with $W=\glsub$, tells us that for all 
$i \in \{1,\dots, \delta\}$, 
$\Psi_{i}(\glsub)\subset \glsub$.
So the abelian Lie-algebra $\glsub$ is endowed with a natural structure of $\CR$-module. 

An $\CR$-module $W$ is called \emph{decomposable} if it admits two proper $\CR$-submodules $W_1$ and $W_2$ such that $W= W_1 \oplus W_2$; it is 
\emph{indecomposable} otherwise. A morphism $\phi:W_1 \rightarrow W_2$ is a morphism of $\CR$-modules if each $W_i$ is an $\CR$-module
and, for all $M\in \CR$ and $N\in W_1$, $\phi(M.N)=M.\phi(N)$. We write $W_1 \simeq_{\CR} W_2$ when the $W_i$ are isomorphic $\CR$-modules. 

\begin{proposition}[Krull-Schmidt, \cite{La01a}, Corollary 19.22, page 288]\label{propo3}
The $\CR$-module $\glsub$  admits a decomposition 
	$\glsub =\displaystylee \bigoplus_{i=1}^{\kappa} W_{i},$
such that: 
\begin{itemize}
\item Each $W_{i}$ is an $\CR$-module.
\item Each $W_i$ admits a decomposition $W_{i}=\displaystylee \bigoplus_{j=1}^{\nu_{i}} V_{i,j}$, where the $(V_{i,j})_{1\leq j\leq \nu_{i}}$
 are pairwise isomorphic indecomposable $\CR$-modules.
\item For $i_1 \neq i_2$, all non zero indecomposable $\CR$-modules $V_{i_1,s}\subset W_{i_1}$ and $V_{i_2,t}\subset W_{i_2}$ are non-isomorphic 
$\CR$-modules.
\end{itemize}
Moreover, this decomposition is unique up to $\CR$-module isomorphisms.
\end{proposition}

\begin{remark}
With a standard slight abuse of notations, we may write $W_i\simeq_{\CR} \nu_i V_i$, for some indecomposable $\CR$-module $V_i$.
The numbers $\kappa$ and $\nu_i$, as well as the $\CR$-module isomorphism class of  $V_i$ are uniquely determined in the 
isotypical decomposition.
\end{remark} 

\begin{definition} \label{isotypical blocks}
{The decomposition $\glsub=\displaystylee \bigoplus_{i=1}^{\kappa} W_{i}$ in Proposition \ref{propo3} is called the \emph{isotypical decomposition} of the $\CR$-module $\glsub$. }\\
A maximal direct sum of pairwise isomorphic indecomposable $\CR$-modules, i.e. one of the spaces $W_i$,
is called an \emph{isotypical block} in the isotypical decomposition of $\glsub$.
\end{definition}

Computing an isotypical decomposition is classically achieved by studying the \emph{eigenring} 
\[ \mathrm{End}_{\CR}(\glsub) := \{ M\in \mathrm{End}_{\CC}(\glsub) \, | \, \forall i\in  \{1, \dots, \delta\}, 
\, M.\Psi_i = \Psi_i.M\},
\]
where $\mathrm{End}_{\CC}(\glsub)$ is the algebra of $\CC$-linear endomorphisms of $\glsub$.
We review in the next paragraph how to use $\mathrm{End}_{\CR}(\glsub)$ to compute a decomposition.
This will follow, for example, from Fitting's Lemma in \cite{La01a}, Lemma 19.16, page 285, or \cite{Ba07a}, where the process is described in the context of Ore-modules
and \cite{BaClWe05a}, where the case of several matrices is addressed. 
The following known algorithm (\cite{Ba07a}) computes the 
isotypical decomposition of  $\glsub$.
\begin{trivlist}
\item  \textbf{Input}: the list of matrices $\Psi_i$. 
\item \textbf{Output}: isotypical decomposition of $\glsub$.
\begin{enumerate}
\item Pick a matrix $M$ with indeterminate coefficients. Solving the linear conditions $M.\Psi_i = \Psi_i.M$ gives  a basis of the eigenring $\mathrm{End}_{\CR}(\glsub)$.
\item Pick a ``sufficiently general'' element $P\in \mathrm{End}_{\CR}(\glsub)$ (see \cite{Ba07a}). 
\item Factor its characteristic polynomial 
as $\chi_P(\lambda)=\prod_i{\chi_i(\lambda)^{m_i}}$.
\item For each factor, compute a basis of the generalized eigenspaces $\ker\left(\chi_i(P)^{m_i}\right)$.
\end{enumerate}
\item \textbf{Return}: A matrix $T$ whose columns are bases of the $\ker\left(\chi_i(P)^{m_i}\right)$.
\end{trivlist}

  Note that the invariant subspaces of $P$ are the $\CR$-submodules 
  of $\glsub$
  and $T^{-1}  \Psi T$ is in block diagonal form, where each block is an indecomposable  $\CR$-module.
This process is described in detail in \cite{Ba07a}, see also \cite{PS03}, Proposition 2.40 or \cite{BaSaWe19a}. 
Indeed, as the differential system $[\Psi]$ is in reduced form, see Remark \ref{psi-reduit}, Theorem~1 of \cite{ApCoWe13a} shows that $\mathrm{End}_{\CR}(\glsub)$ is the eigenring of $[\Psi]$ (in the usual sense)  and \cite{Ba07a} applies mutatis mutandis.  
In \cite{BaSaWe19a}, refinements are given on how to use the eigenring structure to compute the isomorphism classes inside each isotypical block; this will be used in problem \textbf{P2} below.

\begin{remark} It would be tempting to use the factors of $\Pi_{\Psi}$, the minimal polynomial of $\Psi$, to compute an isotypical decomposition.
However, this would be the source of mistakes, as the characteristic spaces of $\Psi$ are defined over $\mathbf{k}$, not over $\CC$. There are examples, see $\S \ref{sec:ex}$, where $\glsub$ is indecomposable while $\Pi_{\Psi}$ is the product of several coprime polynomials.
\end{remark}


\subsubsection{The flag decomposition of an indecomposable $\Psi$-space}\label{ondecoupe}
Let $U$ be a $\Psi$-space. We say that $U$ is an \emph{irreducible} $\Psi$-space if its only $\Psi$-subspaces are $\{0\}$ and $U$.
Schur's lemma shows that any automorphism of  an irreducible $\Psi$-space is a scalar multiple of the identity. This is generalized in the following lemma,  which is sometimes known as Goursat's lemma.

\begin{lemma}[Goursat's lemma, \cite{CoSi98a}, Lemma 2.2] \label{lem1} 
Let $V : = U_1 \oplus \cdots \oplus U_{\nu}$ where $U_1, \ldots, U_{\nu}$ denote pairwise
isomorphic {irreducible} $\Psi$-spaces. Let $\widetilde{\phi}_j:U_1\rightarrow U_j$ be an isomorphism from 
$U_1$ to $U_j$. 
{Let $W$ be an irreducible $\Psi$-subspace of $V$. Then, there exist $c_1,\dots c_{\nu} \in \CC$ such that $W=U_{\underline{c}}$, where }
$U_{\underline{c}} := \{\sum_{j=1}^{\nu}c_{j}\widetilde{\phi}_j(u), u\in U_1\}$.
Any $\Psi$-subspace of such a $V$ is a direct sum of modules $U_{\underline{c}}$ as described in the lemma.
\end{lemma}


We now construct a special flag for an isotypical bloc, called a \emph{$\Psi$-isotypical flag}, adapted to our reduction process. We are going to proceed in two steps. 
\\
First, we construct such a flag in the case of an indecomposable $\Psi$-space $V$. 
Let $U_1$ be an irreducible $\Psi$-subspace of $V$. There may be other subspaces which are $\Psi$-isomorphic to $U_1$; let $n_1$ be the maximal number of $\Psi$-subspaces of $V$ whose direct sum is a subspace of $V$ and which are all isomorphic to $U_1$. 
    {
    We define $V^{[1]}$ as this direct sum, so that $V^{[1]}\simeq n_1U_1 \subset V$. 
}
    If $V^{[1]}=V$ we are done. Otherwise, we look at the quotient $V/V^{[1]}$ and apply the same construction: we obtain a
     subspace $\tilde{W}_2:=n_2 U_2$ of $V/V^{[1]}$. Note that if $E_{1},\dots,E_{n}$ denotes a basis of $V$ such that $E_{1},\dots,E_{k}$ is a basis of $V^{[1]}$, then  $V/V^{[1]}$ is isomorphic, as a vector space, to $\mathrm{Vect}_{\CC}(E_{k+1},\dots,E_{n})\subset V$. Then,  we may lift  $\tilde{W}_2$ to a $\Psi$-subspace $V^{[2]}$ of $V$ which contains $V^{[1]}$ and such that $V^{[2]}/V^{[1]}= \tilde{W}_2$. 
      We iterate and the result of this construction is what we call a \emph{$\Psi$-isotypical flag}:
     \[
     V= V^{[\mu]} \supsetneq V^{[\mu-1]}\supsetneq \cdots \supsetneq V^{[1]} \supsetneq V^{[0]}=\{0\},
     \] 
     such that 
     each $V^{[j]}/V^{[j-1]}$
     is a direct sum of  pairwise isomorphic irreducible $\Psi$-spaces. 
     \\\par 
        
{We  now define a $\Psi$-isotypical flag for an  isotypical block $W$ among the $W_i$ given 
by the isotypical decomposition $\glsub =\displaystylee \bigoplus_{i=1}^{\kappa} W_{i}$ of Proposition \ref{propo3}.  }
\\
{ Let $W$ denote an isotypical block in the isotypical decomposition of $\glsub$.} We have a decomposition $W=\displaystylee \bigoplus_{j=1}^{\nu} V_{j}$, 
where the $(V_{j})_{1\leq j\leq \nu}$
 are pairwise isomorphic indecomposable $\CR$-modules. 
 We first construct, as above,  a $\Psi$-isotypical flag for $V_{1}$: 
 \[V_{1}= V_{1}^{[\mu_{1}]} \supsetneq V_{1}^{[\mu_{1}-1]}\supsetneq \cdots \supsetneq V_{1}^{[1]} \supsetneq V_{1}^{[0]}=\{0\},
 \] 
 such that each
 $V_{1}^{[k]}/V_{1}^{[k-1]}$ is a direct sum of several pairwise isomorphic irreducible $\Psi$-spaces. 
 \\ 
 Now let $\phi_j: V_{1}\rightarrow V_{j}$ denote an  $\CR$-module isomorphism. 
 We set
 $V_{j}^{[k]}:=\phi_j(V_{1}^{[k]})$ for all $k$ and this defines a $\Psi$-isotypical flag for $V_{j}$.
 Now, we define $W^{[k]}$ by {$W^{[k]}:=\displaystylee \bigoplus_{j=1}^{\nu} V_{j}^{[k]}$}.
  
 \begin{definition} \label{isotypical-flag}
 Let $W$ denote an isotypical block in the isotypical decomposition of $\glsub$.
 The flag 
	\[
     W= W^{[\mu]} \supsetneq W^{[\mu-1]}\supsetneq \cdots \supsetneq W^{[1]} \supsetneq W^{[0]}=\{0\},
     \]
     constructed above is called an \emph{isotypical flag} (or $\Psi$-isotypical flag) for $W$.
 \end{definition}

     We now discuss how to compute such a $\Psi$-isotypical flag. We have explained how to compute the isotypical decomposition so we start by computing the $\Psi$-isotypical flag of an indecomposable $\Psi$-space. We thus need to be able to solve the  following two problems:
\begin{enumerate}
\item[\textbf{P1:}]{Given a $\Psi$-space $V$, find an irreducible $\Psi$-subspace $U\subseteq V$}.
\item [\textbf{P2:}] If $U\subset V$ is an irreducible  $\Psi$-subspace, determine the maximal $n\in \mathbb{N}^*$ such that $V$ {contains} 
a direct sum of $n$ subspaces $\Psi$-isomorphic to $U$, i.e. 
$nU \subseteq V$.
\end{enumerate}

{In order to address Problem \textbf{P1}, we will use the following notion.}

\begin{definition} \label{associated-psi-space}
Given a subspace $\CW$ of $\glsub(\kk)$ stable under $\Psi$, its \emph{associated $\Psi$-space} is
the smallest $\Psi$-subspace  $W$ of $\glsub$ such that
$\CW \subset W\otimes_\CC \kk$.
\end{definition}
Computation of the $\Psi$-space $W$ associated to $\CW$
can be achieved as follows. 
\begin{trivlist}
\item  \textbf{Input}: $\CW$, a subspace of $\glsub(\mathbf{k})$.
\item \textbf{Output}: the associated $\Psi$-space $W$.
\begin{enumerate}
\item For each element of a basis of $\CW$, compute its Wei-Norman decomposition.
\item Compute a basis $\CB$ of the orbits under $\mathcal{R}$ of all elements of these Wei-Norman decompositions.
\item The $\Psi$-space $W$ is the vector space generated by $\CB$.
\end{enumerate}
\end{trivlist}
Let $V$ denote a $\Psi$-subspace of $\glsub$. We now show how to find an irreducible $\Psi$-subspace $U\subseteq V$.

We apply the eigenring method for the isotypical decomposition as above; we let $U$ be an indecomposable subspace in the decomposition (or $U=V$ if $V$ is indecomposable).
We identify $\Psi$ with its restriction to $U\otimes_{\CC} \mathbf{k}$ in this paragraph. 
Let $\chi_\Psi(\lambda)$ denote its characteristic polynomial. 
If $\tilde{U}$ is an irreducible $\Psi$-subspace of $U$, then 
the characteristic polynomial of the restriction of $\Psi$ to $\tilde{U}\otimes_{\CC} \mathbf{k}$ divides $\chi_\Psi(\lambda)$. We compute a factorization
${\chi_\Psi(\lambda) = f_1(\lambda)^{m_1} \cdots f_d(\lambda)^{m_d}}$
where the $f_i$ are pairwise coprime irreducible polynomials over $\mathbf{k}$.
For each $i$, we compute $E_i := \ker( f_i(\Psi) )$; then we compute the $\Psi$-space\footnote{Note that $E_i$ is a vector space over $\mathbf{k}$ whereas $W_i$ is a vector space over $\CC$; even though $E_i$ is an irreducible subspace of $U\otimes_{\CC} \mathbf{k}$,  the space $W_i$ may still be a reducible $\Psi$-space: see the $B3\times B2$ Example in $\S\ref{B3xB2}$. } $W_i \subset \glsub$ associated to $E_i$. 
If all $W_i$ are equal to $U$ then $U$ is an irreducible $\Psi$-space
and we return $U$. Otherwise, pick a $W_i$ of minimal dimension and repeat the above steps with $W_i$ in place of $U$ (eigenring, generalized eigenspaces, $\Psi$-space).
\par
The dimension  decreases strictly at each step so the process terminates and produces an irreducible $\Psi$-subspace $U\subseteq V$.
 \begin{trivlist}
\item  \textbf{Input}: $V$, a $\Psi$-subspace of $\glsub$.
\item \textbf{Output}: an irreducible $\Psi$-space $U\subset V$.
\begin{enumerate}
\item Compute an isotypical decomposition. Let $U$ denote one of the indecomposable subspaces.
\item Factor the characteristic polynomial of $\Psi$ on $U\otimes_{\CC} \mathbf{k}$: 
$\chi_\Psi(\lambda) = f_1(\lambda)^{m_1} \cdots f_d(\lambda)^{m_d} $
where the $f_i$ are coprime irreducible polynomials.
\item Compute the $E_i:=\ker\left(f_i(\Psi)\right)
$ and the associated $\Psi$-space $W_i$.
\item If all $W_i$ are equal to $U$, then $U$ is irreducible: return $U$.
\\
Otherwise, apply recursively this procedure to a $W_i$ of minimal dimension. 
\end{enumerate}
\end{trivlist}

Now let us consider Problem \textbf{P2}.
Assume that we have found an irreducible $\Psi$-subspace $U$  of $V$. 
  Finding $\CR$-submodules of $V$ which are $\Psi$-isomorphic to $U$ amounts to finding elements in 
     $\mathrm{Hom}_{\mathcal R}(U, V)$.
Let  $s$ be the dimension of $U$ and $d$ be the dimension of $V$.  We have $s\leq d$. The matrix of $\Psi$ restricted to $V$ is of the form
\[ \Psi|_V = 
	\left(\begin{array}{c|c} \Phi & \star  \\\hline 
0 & \star \end{array}\right),
 \]
 where $\Phi$ is a square matrix of size $s$ representing $\Psi|_U$. 
 Let $\Psi_j$  denote  the constant matrices in a 
 Wei-Norman decomposition of $\Psi|_V$. They induce matrices  $\Phi_{j}$  that generate a
 Wei-Norman decomposition of $\Phi$. Elements of 
     $\mathrm{Hom}_{\mathcal R}(U, V)$ are represented by matrices 
     $L \in \CM_{d \times s}(\CC)$ such that, for all $j$, 
     we have $\Psi_j\cdot L  =L\cdot  \Phi_j   $.
     This gives a linear system of equations for the entries of $L$.
     Once a basis $L_1, \ldots, L_m$ of these $L$ is found, we let
     \[ P:= \left(
\begin{array}{c|c|c|c}   
&&&\\ 
    L_1 &\dots & L_m &0\\
    
  &&&\mathrm{Id}_{d-sm}   \end{array} 
     \right). \]
     The conjugation  given by $P^{-1}\Psi|_V  P$
	puts $\Psi|_V $ in a form where the north-west block is a direct sum of $m$ copies of $\Phi$.

 \subsection{Examples of Decomposition}\label{sec:ex}
In this subsection, we compute the isotypical decomposition and the flags  with the desired properties in several examples.  
{A Maple worksheet\footnote{The reader may also find a pdf version at\\ \url{http://www.unilim.fr/pages_perso/jacques-arthur.weil/DreyfusWeilReductionExamples.pdf}} with these examples may be found at \cite{DrWe20a}.}
In what follows, the $E_{i,j}$ are the elementary matrices forming the canonical basis of $\mathcal{M}_n$, i.e. $E_{i,j}$ has a $1$ on the $(i,j)$ entry and $0$ elsewhere.  
We first focus on the isotypical decomposition ; we will first expose our reduction technique on these examples as we believe that it may help the reader when we establish the theory in $\S \ref{sec4}$. In each of the five examples below, we compute the isotypical decomposition using only the block-diagonal part of systems which will be fully be written down in $\S \ref{sec:exred}$. 

\subsubsection{The ``$SO_3\times SL_2$'' Example.}\label{SO3xSL2}

We consider a system whose diagonal part is given by 
$$ A_{diag} (x) := 
\left( \begin {array}{ccc|cc} 0&1&x&0&0\\ -1&0&0&0&0
\\ -x&0&0&0&0\\ \hline 0&0&0&0&1
\\0&0&0&-x&0\end {array} \right) . 
$$
This matrix is the block-diagonal part of the system studied later in Section~\ref{SO3xSL2-continued}.
The diagonal blocks are in the Lie algebras $\mathfrak{so}_3$ of the 3-dimensional special orthogonal group  and 
$\mathfrak{sl}_2$ of the special linear group.

The matrix $A_{diag} (x)$ is in reduced form 
and its associated Lie algebra is of dimension $6$, as we may see using \cite{ApCoWe13a,BaClDiWe16a}.
In this example and the following one, $(B_{i})_{1\leq i\leq n_1 n_2}$, denotes the canonical basis of $\glsub$, i.e. in this example $B_1:=E_{4,1}$, $B_2:=E_{4,2}$, $B_3:=E_{4,3}$, $B_4:=E_{5,1}$, $B_5:=E_{5,2}$, $B_6:=E_{5,3}$.

Using  this basis $\{ B_1, \ldots, B_6\}$ of $\glsub$, 
we find the corresponding matrix of the adjoint action, given by
$$ \Psi = 
\left(
\begin {array}{cccccc} 0&1&x&1&0&0\\  -1&0&0&0
&1&0\\  -x&0&0&0&0&1\\  -x&0&0&0&1&x
\\  0&-x&0&-1&0&0\\  0&0&-x&-x&0&0
\end {array}
 \right).
$$
The eigenring contains only the identity, which shows  that $\glsub$ is $\Psi$-indecomposable.
The characteristic polynomial of $\Psi$ has two factors.
$$\chi_{\Psi}(\lambda) =  \left( {\lambda}^{2}+x \right)  \left( {\lambda}^{4}+2\,{\lambda}^{2}
{x}^{2}+{x}^{4}+2\,{\lambda}^{2}x-2\,{x}^{3}+2\,{\lambda}^{2}+3\,{x}^{
2}-2\,x+1 \right).
$$

The corresponding generalized eigenspaces (over $\mathbf{k}$) are:
\[
E_1= \left<  \left( \begin {array}{c} 0\\  0
\\  0\\  0\\  -x
\\  1\end {array} \right) , \left( \begin {array}{c} 0
\\  -x\\  1\\  0
\\  0\\  0\end {array} \right) 
 \right> 
 \textrm{ and }
 E_2 = \left<  \left( \begin {array}{c} 0\\  0
\\  0\\  0\\ 
  \frac{1}{x}\\  1\end {array} \right) , \left( \begin {array}
{c} 0\\  0\\  0\\  
1\\  0\\  0\end {array} \right) ,
 \left( \begin {array}{c} 0\\  \frac{1}{x}
\\  1\\  0\\  0
\\  0\end {array} \right) , \left( \begin {array}{c} 
1\\  0\\  0\\  0
\\  0\\  0\end {array} \right) 
 \right> .
\]
For each of them, the associated $\Psi$-space, see Definition \ref{associated-psi-space}, is the whole $\glsub$.
We conclude that $\glsub$ is $\Psi$-irreducible. 
The ``flag'' only has one level in this case:

\begin{center}\begin{tikzpicture}
\draw[rounded corners] (-2.5, -0.5) rectangle (2.5, 0.5);
\draw (-3.5,0) node{$\glsub$} ;
\draw (-3,0) node{$=$} ;
\draw (0,0) node{$\langle B_{1},B_2,B_3,B_4,B_5,B_6 \rangle$} ;
\end{tikzpicture}\end{center}

\begin{remark}
In the spirit and notations of Remark \ref{psi-reduit}, we note that $\CM_1$ and $\CM_2$ are irreducible modules so $\CM_1^{\star}\otimes \CM_2$ is a completely reducible module. 
As $\glsub$ is an indecomposable $\Psi$-space, this shows that it is actually irreducible.
\end{remark}

\subsubsection{The ``$SO_3\times B_2$'' Example.}\label{SO3xB2}

Consider the system $[A_{diag}]$ given by:
$$ A_{diag} (x) := \left( \begin {array}{ccc|cc} 0&1&x&0&0\\ -1&0&0&0&0
\\ -x&0&0&0&0\\ \hline 0&0&0&x&1
\\  0&0&0&0&-x\end {array} \right) 
= \left( \begin {array}{c|c} A_1(x) & 0 \\ \hline 0 & A_2(x) \end {array} \right).
$$
The diagonal blocks are respectively in the Lie algebras $\mathfrak{so}_3$ of the 3-dimensional special orthogonal group  and 
 in the Lie algebra 
	$$\mathfrak{b}_2 :=\mathrm{Vect}_{\CC}\left\{ 
		\left(\begin{array}{cc}1 & 0 \\0 & -1 \end{array}\right) , 
		\left(\begin{array}{cc}0 & 1 \\0 & 0\end{array}\right) \right\}$$ of the two-dimensional Borel group.
The matrix $A_{diag} (x)$ is in reduced form as we may see using \cite{ApCoWe13a,BaClDiWe16a}. 
\\
Using the canonical basis of $\glsub$ as previously, the matrix for the adjoint action 
${\Psi=[\adiag,\bullet]}$, acting on $\glsub$, is
	$$ \Psi = \left(
\begin {array}{ccc|ccc} 
	x&1&x&1&0&0\\ 
	-1&x&0&0 &1&0\\ 
	-x&0&x&0&0&1\\ 
\hline	0&0&0&-x&1&x \\ 
	0&0&0&-1&-x&0\\ 
	0&0&0&-x&0&-x
\end {array} 
 \right)  =: \Psi_0 + x \Psi_1.
$$
As above, we let $\CR:=\CC[\Psi_0,\Psi_1]$. Computation shows that the eigenring $\text{End}_\CR(\glsub)$ is spanned by the identity.
So $\glsub$ is $\Psi$-indecomposable.
Looking at the matrix $\Psi$, we immediately see that the space spanned by the first three vectors is a $\Psi$-space. Let us recover that 
{using the algorithm in 	\S \ref{ondecoupe}
}
 to illustrate the method.\par

If a subspace {$V\subset \glsub$} is a $\Psi$-space in $\glsub$ then $V\otimes_{\CC} \mathbf{k}$ is invariant. Such an invariant subspace is found from the generalized eigenspaces of $\Psi$.
The characteristic polynomial $\chi_{\Psi}(\lambda)$ of $\Psi$ has four factors
$f_1(x)=\left( \lambda -x \right)$, 
$f_2(x) = \left( \lambda+x \right)$ , 
$f_3(x) = \left( {\lambda}^{2}-2\,\lambda\,x+2\,{x}^{2}+1 \right)$ 
and 
$f_4(x) = \left( {\lambda}^{2}+2\, \lambda\,x+2\,{x}^{2}+1 \right)$.
The corresponding generalized eigenspaces in 
{$\glsub\otimes_{\CC} \mathbf{k}$}
are respectively 
\[
\begin{array}{ll}
E_1 = \left<  \left( \begin {array}{c} 0\\  -x
\\  1\\  0\\  0
\\  0\end {array} \right) \right> , &
E_2=\left<
 \left( \begin {array}{c} 0\\  \frac12
\\  -\,\frac{1}{2x}\\  0
\\  -x\\  1\end {array} \right) 
 \right> ,\\
 E_3= \left< \left( \begin {array}{c} 0\\  \frac1x \\  1\\  0
\\  0\\  0\end {array} \right),
 \left( \begin {array}{c} 1\\  0\\  0
\\  0\\  0\\  0
\end {array} \right)  \right>, & E_4= \left< \left( \begin {array}{c} 0\\  \frac{-1}{2x^2} \\ \frac{-1}{2x}\\  0
\\  \frac{1}{x}\\  1\end {array} \right),
 \left( \begin {array}{c} \frac{-1}{2x}\\  0\\  0
\\  1\\  0\\  0
\end {array} \right)  \right> .
\end{array}
\]
We compute the smallest subspace $\overline{V}_1$ of $\glsub$ such that $E_1=\overline{V}_1\otimes_\CC \mathbf{k}$: it is found from a Wei-Norman decomposition of the generator of $E_1$. 
Now we let $V_1$ be the orbit of $\overline{V}_1$ under $\CR$. We find that 
$V_1=<B_1, B_2, B_3>$. Proceeding similarly with $E_2$, $E_3$ and $E_4$, we find respectively $V_3=V_1$ and $V_2=V_4=\glsub$.
 Note that $V_1\otimes_\CC \mathbf{k} = E_1 \oplus E_3$.
 As the dimension of $V_1$ is minimal, it is $\Psi$-irreducible.
 
 We let $B_i^{[1]}:=B_i$, for $i=1,2,3$, $W^{[1]}:=V_1$ and then 
$B_1^{[2]},B_2^{[2]},B_3^{[2]} = B_4,B_5,B_6 $
 to obtain the flag $\glsub = W^{[2]}\supsetneq W^{[1]}   \supsetneq \{0\}$:

\begin{center}
\begin{tikzpicture}
\draw[rounded corners] (-1.75, -2) rectangle (1.75, 2);
\draw[thick,->](0,0.5)--(0,-0.5);
\draw (0,1) node[draw,rounded corners]{$\langle B_1^{[2]}, B_2^{[2]}, B_3^{[2]} \rangle$} ;
\draw (0,-1) node[draw,rounded corners]{$\langle B_1^{[1]}, B_2^{[1]}, B_3^{[1]} \rangle$} ;
\draw (-2.5,0) node{$\glsub$} ;
\draw (-2,0) node{$=$} ;
\end{tikzpicture}
\end{center}

This example is continued in \S \ref{SO3xB2-continued}.

\subsubsection{The ``$B_3\times B_2$'' Example.}\label{B3xB2}

Let us consider
$$ A_{diag} (x) := 
\left( \begin{array}{ccc|cc} 1&x&0&0&0\\ 0&-x-1&0&0
&0\\ 0&0&x&0&0\\ \hline 0&0&0&x&1
\\ 0&0&0&0&-x\end{array} \right) .
$$
The associated Lie algebra has dimension $4$ and it turns out that this system is in reduced form.
First, the diagonal $\mathrm{Diag}(1,-x-1,x,x,-x)$ is in reduced form : its associated Lie algebra has dimension $2$ while the associated Picard-Vessiot extension is generated over $\C(x)$ by $e^x$ and $e^\frac{x^2}{2}$, which are algebraically independent ; 
the remaining reduction 
(applying the full algorithm at the end of this paper) is a simple integration exercise\footnote{This gives,  in turn,  a  proof that $e^x$, $e^\frac{x^2}{2}$, $\int^x{e^\frac{t^2}{2}}\,{\rm d}t$ 
and $\int^x {{\rm e}^{- \frac12\left( t+2 \right) ^{2}}}
\,{\rm d}t$ are algebraically independent.
}
.

The matrix of the adjoint action $\Psi = [\adiag,\bullet]$ in the canonical basis is
$$ \Psi = 
\left( \begin {array}{cccccc} x-1&0&0&1&0&0\\  -x&2
\,x+1&0&0&1&0\\  0&0&0&0&0&1\\  0&0&0
&-x-1&0&0\\  0&0&0&-x&1&0\\  0&0&0&0
&0&-2\,x\end {array} \right) .
$$
Using the eigenring decomposition algorithm from \S \ref{ondecoupe},
we find a decomposition ${\glsub=W_1\oplus W_2}$ as a direct sum of two indecomposable subspaces $W_i$ of respective dimensions $4$ and $2$.
The restrictions of $\Psi$ to these subspaces have respective matrices
\[ 
\Psi|_{W_1} = 
\left( 
\begin {array}{cccc} 1&-x&0&0\\  0&-1-x&0&0
\\  1&0&2\,x+1&-x\\  0&1&0&x-1
\end {array}
\right)
\textrm{ and } 
\Psi|_{W_2} = 
\left(
 \begin {array}{cc} 0&1\\  0& -2x\end {array}
 \right).
\]

We have $W_{1}=\langle C_1,C_2,C_3,C_4 \rangle$ and $W_{2}=\langle C_5,C_6 \rangle$, with $C_1=B_5$, $C_2=B_4$, $C_3=B_2$, $C_4=B_1$, $C_5=B_6$, $C_6=B_3$.
The characteristic polynomial of $\Psi|_{W_1}$ is $\left( \lambda-1 \right)  \left( -2\,x-1+\lambda \right)  \left( -x+1
+\lambda \right)  \left( x+1+\lambda \right)
$. 
For the factor $f_1(\lambda):= \lambda -2\,x-1$, the eigenspace is 
$V_1:= E_1 = \left< C_3\right>.$
It has a constant basis and hence its generator spans a $\Psi$-space.
For the factor $f_2(\lambda):= \lambda-1$, we have 
$ E_2 = \left< -2\,x\,C_1+C_3 \right>.$
The associated $\Psi$-space is 
\[V_2:=\left< C_1, C_3
\right>.\]
{Note that $V_2$ is a \emph{reducible} $\Psi$-space, even though $E_2$ was an irreducible $\mathbf{k}[\Psi]$-module:
}
we have $V_1 \subsetneq V_2$.
Continuing in this way, we find a basis for the flag on $W_1$ : 
\[ B_1^{[1]}:=C_3, B_1^{[2]}:=C_1, B_1^{[3]}:=C_3+C_4, B_1^{[4]}:=C_2. \]
Similarly, the flag on $W_2$ is given by $B_1^{[2]}:=C_5$, $B_1^{[1]}:=C_6$.
The matrix of $\Psi$ in this new basis is 
\[ \left( \begin {array}{cccc|cc} 
	2\,x+1&1&-x&0&0&0\\ 
	0&1&0&-x&0&0\\ 
	0&0&x-1&1&0&0\\ 
	0&0&0 &-x-1&0&0\\ \hline
0&0&0&0&0&1\\ 
0&0&0&0&0
&-2\,x\end {array} \right).
\]
To summarize, our isotypical flag\footnote{We stress the fact that there is another possible choice of flag in this example.  All choices are equivalent, by the Krull-Schmidt theorem, see Proposition \ref{propo3}, so our choice is essentially cosmetic but does not influence the complexity of the computations.
} 
$\glsub=W_1 \oplus W_2$ in this ``$B_3\times B_2$'' example is given by: 

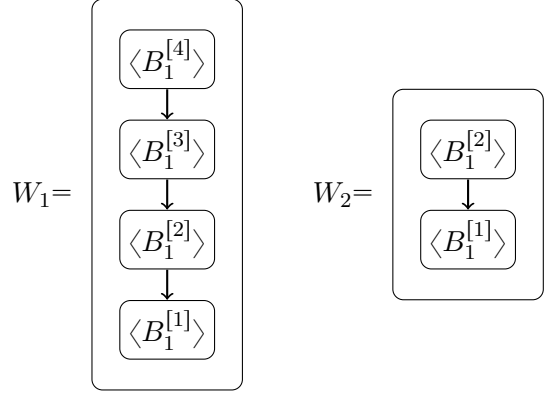
\begin{center}
\begin{tikzpicture}[scale=0.8]
\draw[rounded corners] (4.5, -1.75) rectangle (7, 1.75);
\draw (5.75,-0.75) node[draw,rounded corners]{$\langle B_{1}^{[1]} \rangle$} ;
\draw (5.75,0.75) node[draw,rounded corners]{$\langle B_{1}^{[2]} \rangle$} ;
\draw[thick,->](5.75,0.25)--(5.75,-0.25);

\draw[rounded corners] (-0.5, -3.25) rectangle (2, 3.25);
\draw[thick,->](0.75,1.75)--(0.75,1.25);
\draw[thick,->](0.75,0.25)--(0.75,-0.25);
\draw[thick,->](0.75,-1.25)--(0.75,-1.75);
\draw (0.75,-2.25) node[draw,rounded corners]{$\langle B_{1}^{[1]} \rangle$} ;
\draw (0.75,-0.75) node[draw,rounded corners]{$\langle B_{1}^{[2]} \rangle$} ;
\draw (0.75,0.75) node[draw,rounded corners]{$\langle B_{1}^{[3]} \rangle$} ;
\draw (0.75,2.25) node[draw,rounded corners]{$\langle B_{1}^{[4]} \rangle$} ;
\draw (-1.5,0) node{$W_1$} ;
\draw (-1,0) node{$=$} ;
\draw(3.5,0) node{$W_2$};
\draw (4,0) node{$=$} ;
\end{tikzpicture}\end{center}

\subsubsection{A nilpotent example}\label{nilpotent-example}

Let us consider
$$ \begin{large}A_{diag} (x) := 
\left( \begin {array}{cccc|cccc} 
	1&0&\frac{1}{x}&0&0&0&0&0
	\\  \frac{1}{x-1}&1&0&-\frac{1}{x}&0&0&0&0
	\\ 0&0&1&0&0&0&0&0
	\\ 0&0& \frac{1}{x-1} &1&0&0&0&0
	\\ \hline 0&0&0&0&1&0&\frac{1}{x}&0
\\0&0&0&0& \frac{1}{x-1}&1&0&-\frac{1}{x}
\\0&0&0&0&0&0&1&0\\0&0&0&0&0&0&
 \frac{1}{x-1}&1\end {array} \right).\end{large}
$$
As shown in  the maple worksheet, see \cite{DrWe20a}, or direct computation, see \cite{dreyfus2021differential}, Example~3, the system is in reduced form.
The matrix $\Psi$ of the adjoint action in the canonical basis of $\glsub$ 
is $\Psi:=\frac{1}{x} \Psi_0 + \frac{1}{x-1} \Psi_1$, see the worksheet \cite{DrWe20a}. 
 It turns out that $\Psi$ is nilpotent and that its minimal polynomial is
$\chi_{\Psi}(\lambda) = \lambda^3.$  The eigenring has dimension $32$. 
The eigenring decomposition algorithm provides a decomposition of $\glsub$ as a direct sum of three indecomposable $\Psi$-spaces $W_{1},W_{2},W_{3}$ of respective dimensions $1$, $5$, and $10$.

\begin{figure}[h!]
\begin{tikzpicture}[scale=0.85]
\draw (8,-3) node[draw,rounded corners]{$\langle B_{1}^{[1]} \rangle$} ;
\draw (6,-1.5) node[draw,rounded corners]{$\langle B_{1}^{[2]} \rangle$} ;
\draw (10,-1.5) node[draw,rounded corners]{$\langle B_{2}^{[2]} \rangle$} ;
\draw (5,0) node[draw,rounded corners]{$\langle B_{1}^{[3]} \rangle$} ;
\draw (7,0) node[draw,rounded corners]{$\langle B_{2}^{[3]} \rangle$} ;
\draw (9,0) node[draw,rounded corners]{$\langle B_{3}^{[3]} \rangle$} ;
\draw (11,0) node[draw,rounded corners]{$\langle B_{4}^{[3]}  \rangle$} ;
\draw (6,1.5) node[draw,rounded corners]{$\langle B_{1}^{[4]}  \rangle$} ;
\draw (10,1.5) node[draw,rounded corners]{$\langle B_{2}^{[4]} \rangle$} ;
\draw (8,3) node[draw,rounded corners]{$\langle B_{1}^{[5]} \rangle$} ;
\draw[rounded corners] (3.5,-4) rectangle (12.5, 4);
\draw (3.5,-4) node [above right]{$W_3$};

\draw (6,0) node{$\oplus$} ;
\draw (8,0) node{$\oplus$} ;
\draw (10,0) node{$\oplus$} ;
\draw (8,1.5) node{$\oplus$} ;
\draw (8,-1.5) node{$\oplus$} ;

\draw[rounded corners] (5,1) rectangle (11, 2);
\draw[rounded corners] (4,-0.5) rectangle (12, 0.5);
\draw[rounded corners] (5,-1) rectangle (11, -2);
\draw[thick,->](8,2.5)--(8,2);
\draw[thick,->](8,1)--(8,0.5);
\draw[thick,->](8,-0.5)--(8,-1);
\draw[thick,->](8,-2)--(8,-2.5);

\draw (-1,-1.5) node[draw,rounded corners]{$\langle B_{1}^{[1]} \rangle$} ;
\draw (1,-1.5) node[draw,rounded corners]{$\langle B_{2}^{[1]} \rangle$} ;
\draw (0,0) node[draw,rounded corners]{$\langle B_{1}^{[2]} \rangle$} ;
\draw (-1,1.5) node[draw,rounded corners]{$\langle B_{1}^{[3]} \rangle$} ;
\draw (1,1.5) node[draw,rounded corners]{$\langle B_{2}^{[3]} \rangle$} ;
\draw[rounded corners] (-2.5,-2.5) rectangle (2.5, 2.5);
\draw (-2.5,-2.5) node [above right]{$W_2$};
\draw (0,1.5) node{$\oplus$} ;
\draw (0,-1.5) node{$\oplus$} ;
\draw[rounded corners] (-2,1) rectangle (2, 2);
\draw[rounded corners] (-2,-1) rectangle (2, -2);
\draw[thick,->](0,1)--(0,0.5);
\draw[thick,->](0,-0.5)--(0,-1);

\draw (-4,0) node[draw,rounded corners]{$\langle B_{1}^{[1]} \rangle$} ;
\draw (-4.7,-0.5) node [below right]{$W_1$};

\draw (-2.9,0) node{$\oplus$} ;
\draw (3,0) node{$\oplus$} ;
\end{tikzpicture}

\caption{The isotypical flag of the nilpotent example.
}
\end{figure}
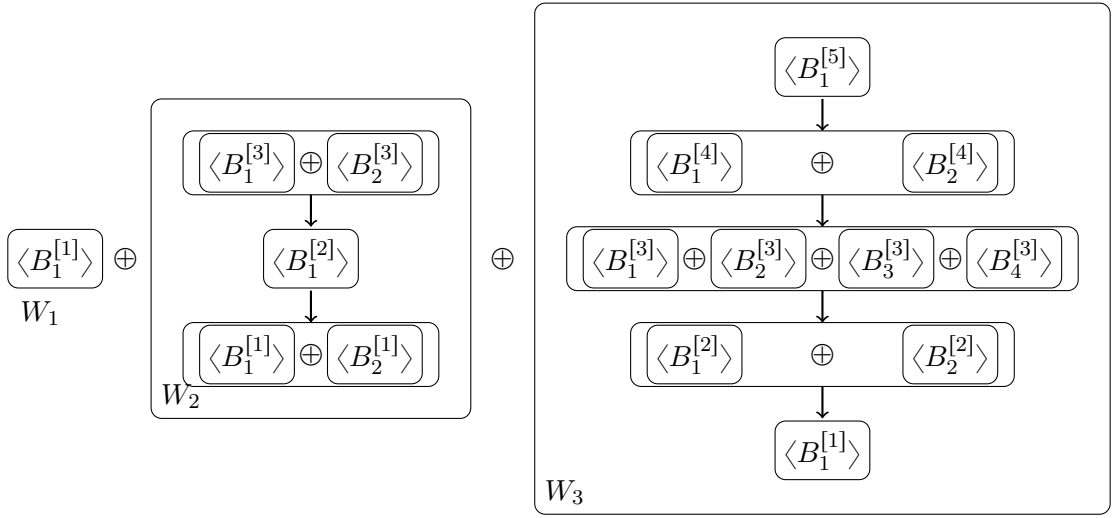

The matrix of $\Psi$ acting on the $5$-dimensional block $W_{2}$  is given in \cite{DrWe20a}.
The flag reduction method provides a new basis on which the matrix of $\Psi|_{W_2}$ is 
\[
\Lambda_2 =  \left(  \begin {array}{cc|c|cc} 
	0&0&\frac{1}{x-1}&0&0
	\\ 0&0&\frac{1}{x}&0&0
	\\ \hline  0&0&0& \frac{1}{x} &\frac{1}{x-1}
	\\  \hline 0&0&0&0&0
	\\ 0&0&0&0&0\end {array} \right)
\]

We note that, although $W_{2}$ is an indecomposable $\Psi$-space, it has decomposable quotients and subspaces, namely
$W_{2}  / \mathrm{Vect}_{\CC}\{ B_1^{[2]} ,B_1^{[1]},  B_2^{[1]} \}$ and 
	$\mathrm{Vect}_{\CC}\{ B_1^{[1]},  B_2^{[1]}\}$ 
	are decomposable.

Similarly, the 
flag reduction method provides  a new basis on which the matrix of $\Psi|_{W_3}$ is 
\[
\Lambda_3= \left( \begin {array}{c|cc|cccc|cc|c}
0&\frac{1}{x}& \frac{1}{x-1}&0&0&0&0&0&0&0
\\ \hline 0&0&0&\frac{1}{x}&- \frac{1}{x-1}&0&0&0&0&0
\\ 0&0&0&0&0&-\frac{1}{x}&
 	\frac{1}{x-1}&0&0&0
 \\ \hline 0&0&0&0&0&0&0&
 	\frac{1}{x-1}&0&0
 \\ 0&0&0&0&0&0&0&0&
 	\frac{1}{x-1}&0
 \\ 0&0&0&0&0&0&0&\frac{1}{x} &0&0
 \\ 0&0&0&0&0&0&0&0&\frac{1}{x}&0
\\ \hline 0&0&0&0&0&0&0&0&0& \frac{1}{x-1}
\\ 0&0&0&0&0&0&0&0&0&\frac{1}{x}
\\ \hline 0&0 &0&0&0&0&0&0&0&0
\end {array} \right).
 \]

\section{Examples of Reduction on an Isotypical Flag}\label{sec:exred}

As the next part of the  algorithm is a bit technical and may be cumbersome to read, we start by performing our reduction technique on the above four examples. 
They are presented in increasing order of  complexity. They are chosen so that the phenomena can be better understood before stating the general reduction procedure.

\subsection{The ``$SO_3\times SL_2$'' Example Continued.} \label{SO3xSL2-continued}
This example is our simplest. Note that the Berman-Singer algorithm \cite{BeSi99a,Be02a} applies to this example (and leads to the same conclusion). Let 
\[ 
	A(x):= 
\left( \begin {array}{ccc|cc} 0&1&x&0&0\\ -1&0&0&0&0
\\ -x&0&0&0&0\\ 
\hline 
 -\frac{1}{x} + \frac{1}{x-1} & 1-\frac{1}{x^2} & x 
&0&1
\\ 
x+ \frac{1}{(x-1)^2} & 1-\frac{1}{x-1} & -1-\frac{1}{x-1}
&-x&0\end {array} \right). \; 
\]
 In $\S\ref{SO3xSL2}$, we have seen that $\glsub$ has dimension $6$ and is $\Psi$-irreducible. 
We recall the notation $\gsub:=\glsub \cap \glie$.
 Since $\glsub$  has no proper $\Psi$-space here, 
we have either $\gsub=\glsub$ 
 (in which case $A(x)$ is already in reduced form) or $\gsub=\{0\}$ 
 (in which case $A_{diag} (x)$ is a reduced form of $[A(x)]$, as we had assumed that $\adiag(x)$ was in reduced form).
 
We look for a reduction matrix of the form 
$P=\id + \sum_{i=1}^6 f_i(x) B_i$ such that ${P[A]=\adiag}$. Writing down this equality, see Proposition \ref{propo2}, we find that the vector of coefficients 
	$\vec{F}:=\left( f_1(x), \ldots, f_6(x) \right)^T$ must be a rational solution of the system  
\[ Y'=\Psi.Y + \vec{b}, \; \textrm{ where } \; \vec{b}:=  \left( \begin {array}{c} 
-\frac{1}{x} + \frac{1}{x-1}
\\ 
1-\frac{1}{x^2}\\ 
 x
\\ 
x+\frac{1}{(x-1)^2}\\
1-\frac{1}{x-1}\\ 
-1- \frac{1}{x-1}
 \end {array} \right).
 \]

Using the \textsc{Maple} implementation of the Barkatou algorithm in the package \texttt{IntegrableConnections}\footnote{The Maple command is \texttt{RationalSolutions([Psi],[x],['rhs',[B]]);}} from \cite{IntegrableConnections}, we find a unique rational solution
$$ \vec{F}= \left( \begin {array}{c} 1\\   \frac{1}{x}
\\  0\\  -\frac{1}{x-1}
\\  0\\  0\end {array} \right)
$$  
and it follows that $\adiag (x)$ is a reduced form of $[A(x)]$ with reduction matrix equal to 
$$P(x):= 
\left( \begin {array}{ccc|cc} 1&0&0&0&0\\ 0&1&0&0&0
\\ 0&0&1&0&0\\ 
\hline 
1 & \frac{1}{x} &0
&1&0
\\ 
- \frac{1}{x-1} & 0& 0&0&1\end {array} \right).$$

\subsection{The ``$SO_3\times B_2$'' Example, Continued} \label{SO3xB2-continued}
We continue with the example from $\S\ref{SO3xB2}$.
$$ A_{diag} (x) := \left( \begin {array}{ccc|cc} 0&1&x&0&0\\ -1&0&0&0&0
\\ -x&0&0&0&0\\\hline  0&0&0&x&1
\\ 0&0&0&0&-x\end {array} \right), \; A_{sub} (x) := \left( \begin {array}{ccc|cc} 0&0&0&0&0\\ 0&0&0&0&0
\\ 0&0&0&0&0
\\ \hline x+3&0&-1&0&0
\\ -{x}^{2}-3\,x&{x}^{2}&{x}^{2}+1&0&0
\end {array} \right).
$$
In $\S\ref{SO3xB2}$, we have found that $\glsub$ has dimension $6$, that it is indecomposable, that it admits only one proper subspace and 
a flag $\glsub = W^{[2]}\supsetneq W^{[1]}   \supsetneq \{0\}$:
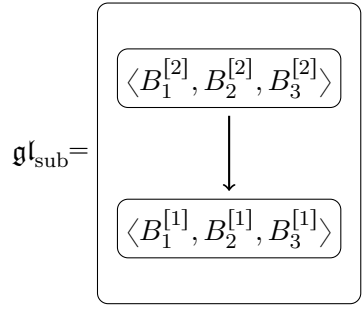
\begin{center}
\begin{tikzpicture}
\draw[rounded corners] (-1.75, -2) rectangle (1.75, 2);
\draw[thick,->](0,0.5)--(0,-0.5);
\draw (0,1) node[draw,rounded corners]{$\langle B_1^{[2]}, B_2^{[2]}, B_3^{[2]} \rangle$} ;
\draw (0,-1) node[draw,rounded corners]{$\langle B_1^{[1]}, B_2^{[1]}, B_3^{[1]} \rangle$} ;
\draw (-2.5,0) node{$\glsub$} ;
\draw (-2,0) node{$=$} ;
\end{tikzpicture}
\end{center}

This means that the only proper $\Psi$-subspace of $\glsub$ is $W^{[1]}$. So there are three possibilities for the reduced matrix.
We start by trying to perform reduction on the first level of the flag, namely $W^{[2]}/W^{[1]}$. 
We look for a gauge transformation of the form $P^{[2]}=\id + \sum_{i=1}^3 f_i(x) B_i^{[2]}$. There is a (partial) reduction if and only if we can find $f_i(x)\in \mathbf{k}$ such that $P^{[2]}[A]$ has no component in 
$W^{[2]}/W^{[1]}$. This means that $\vec{F}:=\left( f_1(x),  f_2(x), f_3(x) \right)^T$ must be a rational solution of the linear differential system
{
\[
Y'  =
\left( \begin {array}{ccc} -x&1&x\\  -1&-x&0\\  -x&0&-x\end {array} \right)
Y
 + \left( \begin {array}{c} - x^2 - 3 x \\   x^2 \\     x^2 +1\end {array} \right).
 \]}

Using again \texttt{IntegrableConnections} from \cite{IntegrableConnections}, we {find a unique  rational solution
 and we derive an intermediate reduction matrix 
$P^{[2]}$ given by 
\[ 
P^{[2]}(x):= \left( 
\begin {array}{ccc|cc} 1&0&0&0&0\\ 0&1&0&0&0
\\ 0&0&1&0&0\\ \hline 0&0&0&1&0
\\ -1&x&x+1&0&1\end {array}
 \right). 
 \]
 }
We let  $A^{[2]}(x) := P^{[2]}(x)[A (x)]=\left( 
\begin {array}{ccc|cc} 0&1&x&0&0\\-1&0&0&0&0
\\ -x&0&0&0&0\\ \hline 2+x&x&x&x&1
\\ 0&0&0&0&-x\end {array}
 \right)$. 
 We now try to reduce the level $W^{[1]}$ of the flag. So we look for a gauge transformation of the form ${P^{[1]}=\id + \sum_{i=1}^3 f_i(x) B_i^{[1]}}$. As $W^{[1]}$ is irreducible, there will be (partial reduction) if and only if $P^{[1]}[A^{[2]}(x)]$ has no components in $W^{[1]}$.
 This means that $\vec{F}:=\left( f_1(x),  f_2(x), f_3(x) \right)^T$ must be a rational solution of the linear differential system
\[ 
  Y'
  =\left( 
\begin {array}{ccc} x&1&x\\  -1&x&0
\\  -x&0&x\end {array}
\right)
\, Y
+
\left( 
\begin {array}{c} 2+x\\  x
\\  x\end {array} 
 \right).
 \]
 As there is no such rational solution,
we find that $A^{[2]}$ cannot be reduced any further so that it is in reduced form. Computing $\liealg(A^{[2]})$ shows that it has dimension $8$.\par  
To summarize,
the flag of the ``$SO_3\times B_2$'' example after the reduction is as follows; the red rectangles correspond to the part we have deleted via the reduction matrix, and the blue rectangles correspond to the non removable part, i.e. the reduced matrix:
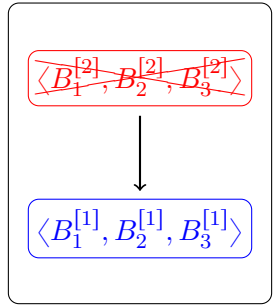
\begin{center}
\begin{tikzpicture}
\draw[rounded corners] (-1.75, -2) rectangle (1.75, 2);
\draw[thick,->](0,0.5)--(0,-0.5);
\draw (0,1) node[draw,red,rounded corners]{$\bcancel{\cancel{\langle B_1^{[2]}, B_2^{[2]}, B_3^{[2]} \rangle}}$} ;
\draw (0,-1) node[draw,blue,rounded corners]{$\langle B_1^{[1]}, B_2^{[1]}, B_3^{[1]} \rangle$} ;
\end{tikzpicture}
\end{center}

\subsection{The ``$B_3\times B_2$'' Example Continued.} \label{B3xB2-continued}
We continue with the example from $\S\ref{B3xB2}$. We have
$$ A_{diag} (x):= 
\left( 
\begin {array}{ccc|cc} 1&x&0&0&0\\ 0&-x-1&0&0
&0\\ 0&0&x&0&0\\ \hline 0&0&0&x&1
\\ 0&0&0&0&-x\end {array}
\right), \; 
A_{sub} (x):=\left( 
\begin {array}{ccc|cc} 0&0&0&0&0\\ 0&0&0&0&0
\\ 0&0&0&0&0\\ 
\hline
-x&1&-1-x&0&0
\\ 
x+1&{\frac {x+1}{x}}&2\,{x}^{2}+1&0&0\end {array}
 \right).
$$
The isotypical flag has the form $\glsub=W_1 \oplus W_2$ with the following flag structures on the $W_i$: 
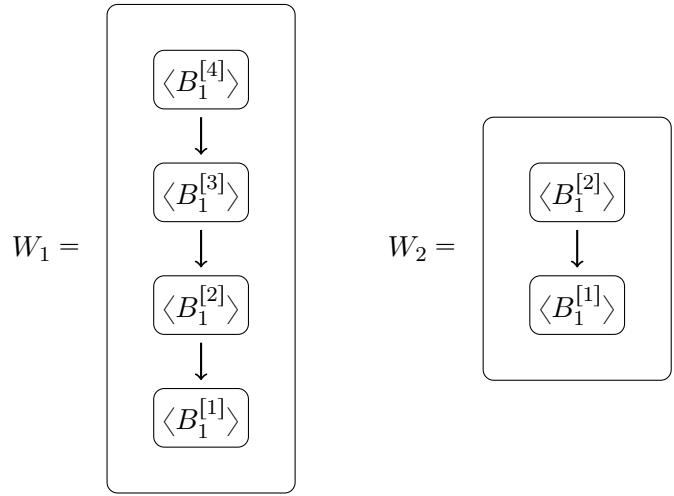
\begin{center}
\begin{tikzpicture}
\draw[rounded corners] (4.5, -1.75) rectangle (7, 1.75);
\draw (5.75,-0.75) node[draw,rounded corners]{$\langle B_{1}^{[1]} \rangle$} ;
\draw (5.75,0.75) node[draw,rounded corners]{$\langle B_{1}^{[2]} \rangle$} ;
\draw[thick,->](5.75,0.25)--(5.75,-0.25);

\draw[rounded corners] (-0.5, -3.25) rectangle (2, 3.25);
\draw[thick,->](0.75,1.75)--(0.75,1.25);
\draw[thick,->](0.75,0.25)--(0.75,-0.25);
\draw[thick,->](0.75,-1.25)--(0.75,-1.75);
\draw (0.75,-2.25) node[draw,rounded corners]{$\langle B_{1}^{[1]} \rangle$} ;
\draw (0.75,-0.75) node[draw,rounded corners]{$\langle B_{1}^{[2]} \rangle$} ;
\draw (0.75,0.75) node[draw,rounded corners]{$\langle B_{1}^{[3]} \rangle$} ;
\draw (0.75,2.25) node[draw,rounded corners]{$\langle B_{1}^{[4]} \rangle$} ;
\draw (-1.5,0) node{$W_1$} ;
\draw (-1,0) node{$=$} ;
\draw(3.5,0) node{$W_2$};
\draw (4,0) node{$=$} ;
\end{tikzpicture}
\end{center}
We take the matrix of $\Psi$ in the adapted basis computed in $\S\ref{B3xB2}$. 

We first perform the reduction on $W_2$, with its adapted flag basis.
We look for a gauge transformation $P^{[2]}:=\id + f_2(x) B_1^{[2]}$
to remove $B_1^{[2]}$ from $P^{[2]}[A_1]$.
We find that $f_2$ should be a rational solution of 
$1-2\,x f_2\left( x \right) +2\,{x}^{2}- f_2' \left( x \right)=0$. The only rational solution is $x$ so 
$P^{[2]}:=\id + x B_1^{[2]}$. Similarly, we look for the last gauge transformation $P^{[1]}:=\id + f_1(x) B_1^{[1]}$.
We find $f_1(x)=-x+c_1$ for an arbitrary constant parameter $c_1$.
So the reduction matrix on $W_2$ is $P_{W_2}=P^{[2]}.P^{[1]}.$
The reduction on $W_2$ is then of the form
$$A_{2}:=P_{W_2}[A]=\left(
 \begin {array}{ccc|cc} 1&x&0&0&0\\ 0&-1-x&0&0
&0\\0&0&x&0&0\\\hline
 -x&1&0&x&1
\\ x+1&	\frac{x+1}{x}&0&0&-x\end {array}
\right)$$
We now perform the reduction on $W_1$, with its adapted flag basis.
We look for a gauge transformation $P^{[4]} = \id + f_4(x) B_1^{[4]}$
so that $ B_1^{[4]}$ would be absent from $P^{[4]}[A]$. The condition is
$ - (1+x) f_4 \left( x \right) +1+x- f_4'\left( x \right) 
=0. $
This equation has the unique rational solution $1$ so we take 
$P^{[4]} = \id + B_1^{[4]}$ and let $A^{[3]}:=P^{[4]}[A_{2}]$.
Now we want to remove $B_1^{[3]}$ from $A^{[3]}$ via 
$P^{[3]} = \id + f_3(x) B_1^{[3]}$.
The condition is $-(1-x)f_3\left( x \right) +1-x -
f_3' \left( x \right)
=0$ which admits the unique rational solution $f_3(x)=1$ so 
$P^{[3]} = \id + B_1^{[3]}$. 
We now set $P^{[2]} = \id + f_2(x) B_1^{[2]}$, look at the condition for $B_1^{[2]}$ to vanish from $P^{[2]}[A^{[2]}]$. We obtain
$f_2' \left( x \right) =f_2 \left( x \right) - {\frac {{x}^{2}-x-1}{x}}$. This equation has no rational solution so we see that we can no further reduce on $W_1$. So we let 
 $P_{W_1}:=P^{[4]}.P^{[3]}$ be the reduction matrix on $W_1$.
\par


\par
Finally, letting $P:=P_{W_1}P_{W_2}$, we find 
\[ A_{\textrm{red}}:=P[A]=\left(
 \begin {array}{ccc|cc} 1&x&0&0&0\\ 0&-1-x&0&0
&0\\0&0&x&0&0\\\hline 0&-x+1&0&x&1
\\ 0&
	-x+1+\frac{1}{x}
&0&0&-x\end {array}
\right).
\]
Note that $A_{\textrm{red}}$ does not depend upon $c_1$.
Now $\liealg(A_{\textrm{red}})$ has dimension $6$ and its off-diagonal part is spanned by the matrices of $B_1^{[1]}$ and $B_1^{[2]}$ of 
$W_1$. Our construction shows that any gauge transformation of the form
$\id + M$ with $M\in \mathrm{Vect}_{\CC} (B_{1}^{[1]} ,B_{1}^{[2]})$ will keep $B_1^{[2]}$ (and then $B_1^{[1]}$) in the Lie algebra so Corollary \ref{coro1} shows that $[A_{\textrm{red}}]$ is in reduced form.\par 
To summarize, the flag of the ``$B_3\times B_2$'' example after the reduction is as follows; the red rectangles correspond to the part we have deleted via the reduction matrix, and the blue rectangles correspond to the non removable part, i.e. the reduced matrix:
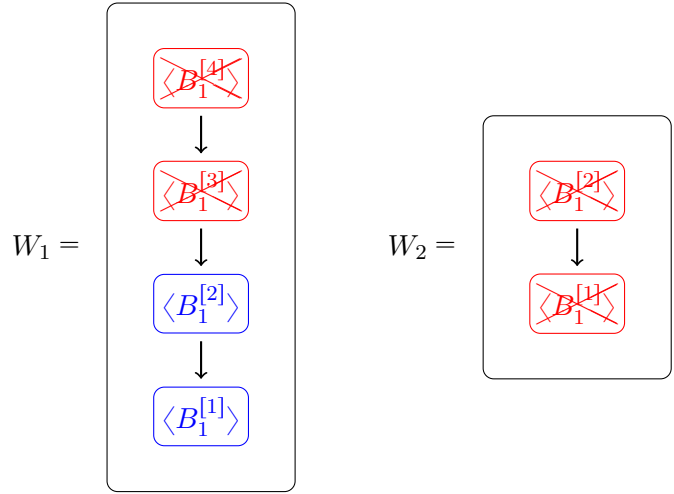
\begin{center}
\begin{tikzpicture}
\draw[rounded corners] (4.5, -1.75) rectangle (7, 1.75);
\draw (5.75,-0.75) node[draw,red,rounded corners]{$\bcancel{\cancel{\langle B_{1}^{[1]} \rangle}}$} ;
\draw (5.75,0.75) node[draw,red,rounded corners]{$\bcancel{\cancel{\langle B_{1}^{[2]} \rangle}}$} ;
\draw[thick,->](5.75,0.25)--(5.75,-0.25);

\draw[rounded corners] (-0.5, -3.25) rectangle (2, 3.25);
\draw[thick,->](0.75,1.75)--(0.75,1.25);
\draw[thick,->](0.75,0.25)--(0.75,-0.25);
\draw[thick,->](0.75,-1.25)--(0.75,-1.75);
\draw (0.75,-2.25) node[draw,blue,rounded corners]{$\langle B_{1}^{[1]} \rangle$} ;
\draw (0.75,-0.75) node[draw,blue,rounded corners]{$
{
{\langle B_{1}^{[2]} \rangle}}$} ;
\draw (0.75,0.75) node[draw,red,rounded corners]{$\bcancel{\cancel{\langle B_{1}^{[3]} \rangle}}$} ;
\draw (0.75,2.25) node[draw,red,rounded corners]{$\bcancel{\cancel{\langle B_{1}^{[4]} \rangle}}$} ;
\draw (-1.5,0) node{$W_1$} ;
\draw (-1,0) node{$=$} ;
\draw(3.5,0) node{$W_2$};
\draw (4,0) node{$=$} ;
\end{tikzpicture}\end{center}

\subsection{The Nilpotent Example, Continued}\label{nilpotent-example-continued}

We continue with the example of $\S\ref{nilpotent-example}$.
We let $(\tilde{N}_i)_{i=1\dots 16}$ denote the adapted basis of $\glsub$ found in $\S\ref{nilpotent-example}$. We recall that we have three indecomposable $\Psi$-spaces of dimensions $1$, $5$ and $10$ respectively. 
We will study 
$$\asub(x) := \sum_{i=1}^{16} f_i(x) \tilde{N}_i ,$$ where the $f_i (x)$ are the following
\[ \begin{array}{llll}
f_{1}(x):=0,&&&\\&&&\\
f_{2}(x):=\frac{1}{x^2},& f_{3}(x):=0,& f_{4}(x):=\frac{2-x}{2x^2},
&f_{5}(x):=\frac{1-x}{x^2},\vspace{0.1cm}\\
	f_6(x):=\frac{3-x}{x^2}, &&\\
	&&&\\
f_{7}(x):=0,&f_{8}(x):=-\frac{1}{2x},&	
	f_{9}(x):= -\frac{1}{2(x-1)},
&f_{10}(x):=\frac{1}{x}, \\
	f_{11}(x):=-\frac{1}{2x},&
	f_{12}(x):=0,
&f_{13}(x):=-\frac{1}{2(x-1)},&
f_{14}(x):=-\frac{1}{2(x-1)},\\
f_{15}(x):=\frac{1}{x^2}-\frac{1}{2(x-1)},&f_{16}(x):=\frac{2}{x^2}+\frac{1}{x-1}.
\end{array}
\]

Since the coefficient in front of $\tilde{N}_1$ is $0$, the reduction to the $1$-dimensional block is already completed.

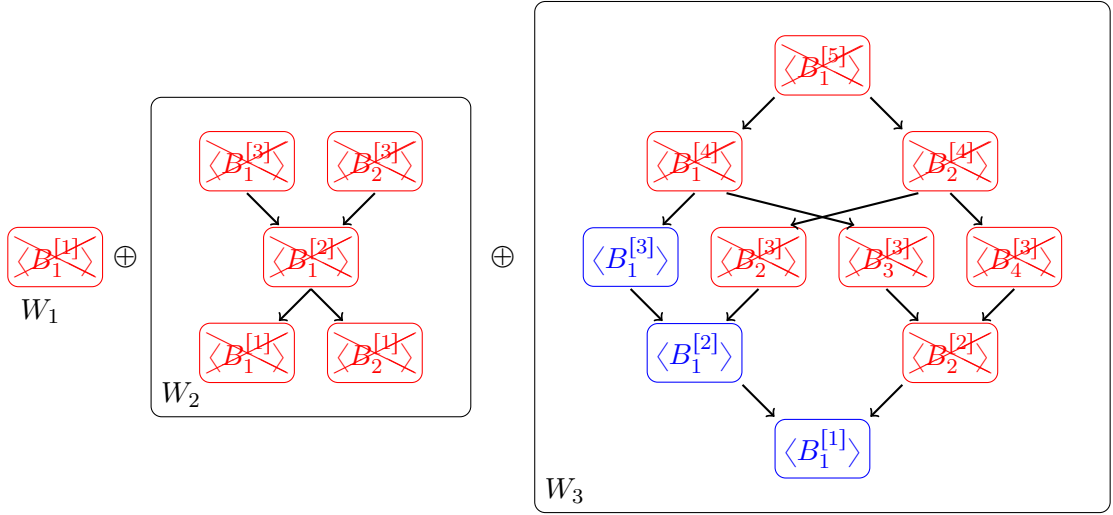
\begin{figure}[h!]

\begin{tikzpicture}[scale=0.85]
\draw (8,-3) node[draw,blue, rounded corners]{$\langle B_{1}^{[1]} \rangle$} ;
\draw (6,-1.5)  node[draw,blue, rounded corners]{$\langle B_{1}^{[2]} \rangle$} ;
\draw (10,-1.5) node[draw,red,rounded corners]{$\bcancel{\cancel{\langle B_{2}^{[2]} \rangle}}$} ;
\draw (5,0)  node[draw,blue, rounded corners]{$\langle B_{1}^{[3]} \rangle$} ;
\draw (7,0) node[draw,red,rounded corners]{$\bcancel{\cancel{\langle B_{2}^{[3]} \rangle}}$} ;
\draw (9,0) node[draw,red,rounded corners]{$\bcancel{\cancel{\langle B_{3}^{[3]} \rangle}}$} ;
\draw (11,0) node[draw,red,rounded corners]{$\bcancel{\cancel{\langle B_{4}^{[3]} \rangle}}$} ;
\draw (6,1.5) node[draw,red,rounded corners]{$\bcancel{\cancel{\langle B_{1}^{[4]}  \rangle}}$} ;
\draw (10,1.5) node[draw,red,rounded corners]{$\bcancel{\cancel{\langle B_{2}^{[4]} \rangle}}$} ;
\draw (8,3) node[draw,red,rounded corners]{$\bcancel{\cancel{\langle B_{1}^{[5]} \rangle}}$} ;
\draw[rounded corners] (3.5,-4) rectangle (12.5, 4);
\draw (3.5,-4) node [above right]{$W_3$};

\draw[thick,->](8.75,2.5)--(9.25,2);
\draw[thick,->](7.25,2.5)--(6.75,2);
\draw[thick,->](6.5,1)--(8.5,0.5);
\draw[thick,->](6,1)--(5.5,0.5);
\draw[thick,->](10,1)--(10.5,0.5);
\draw[thick,->](9.5,1)--(7.5,0.5);
\draw[thick,->](5,-0.5)--(5.5,-1);
\draw[thick,->](9,-0.5)--(9.5,-1);
\draw[thick,->](11,-0.5)--(10.5,-1);
\draw[thick,->](7,-0.5)--(6.5,-1);
\draw[thick,->](9.25,-2)--(8.75,-2.5);
\draw[thick,->](6.75,-2)--(7.25,-2.5);

\draw (-1,-1.5) node[draw,red,rounded corners]{$\bcancel{\cancel{\langle B_{1}^{[1]} \rangle}}$} ;
\draw (1,-1.5) node[draw,red,rounded corners]{$\bcancel{\cancel{\langle B_{2}^{[1]} \rangle}}$} ;
\draw (0,0) node[draw,red,rounded corners]{$\bcancel{\cancel{\langle B_{1}^{[2]} \rangle}}$} ;
\draw (-1,1.5) node[draw,red,rounded corners]{$\bcancel{\cancel{\langle B_{1}^{[3]} \rangle}}$} ;
\draw (1,1.5) node[draw,red,rounded corners]{$\bcancel{\cancel{\langle B_{2}^{[3]} \rangle}}$} ;
\draw[rounded corners] (-2.5,-2.5) rectangle (2.5, 2.5);
\draw (-2.5,-2.5) node [above right]{$W_2$};

\draw[thick,->](-1,1)--(-0.5,0.5);
\draw[thick,->](1,1)--(0.5,0.5);
\draw[thick,->](0,-0.5)--(0.5,-1);
\draw[thick,->](0,-0.5)--(-0.5,-1);

\draw (-4,0) node[draw,red,rounded corners]{$\bcancel{\cancel{\langle B_{1}^{[1]} \rangle}}$} ;
\draw (-4.7,-0.5) node [below right]{$W_1$};
\draw (-2.9,0) node{$\oplus$} ;
\draw (3,0) node{$\oplus$} ;
\end{tikzpicture}
\caption{The action of the adjoint map on the isotypical flag of the nilpotent example  $\S \ref{nilpotent-example-continued}$. The spaces $W_1$ (left), $W_2$ (center) and $W_3$ (right) satisfy $\glsub=W_1\oplus W_2 \oplus W_3$. The red rectangles correspond to the part that we  get rid of via the reduction matrix, and the blue rectangles correspond to what will remain in the reduced matrix.}
\end{figure}

\begin{center}\boxed{\textbf{Reduction of the $5$-dimensional block.}}\end{center}

To remove all of $W_2$, it would be enough to have a rational solution of the system
\[ 
\vec{Y}'=  \Lambda_2.\vec{Y} + \vec{b} \; \textrm{with} \;
{
\Lambda_2 =  \left(  \begin {array}{cc|c|cc} 
	0&0&\frac{1}{x-1}&0&0
	\\ 0&0&\frac{1}{x}&0&0
	\\ \hline  0&0&0& \frac{1}{x} &\frac{1}{x-1}
	\\  \hline 0&0&0&0&0
	\\ 0&0&0&0&0\end {array} \right)
	}
\; \textrm{and} \;
\vec{b}= \left( \begin {array}{c}  
\frac{2}{x} \\ {\frac {2\,x+1}{{x}^{2}}} \\ {\frac {x+1}{{x}^{2}}} \\ 0 \\\frac{1}{x^2} 
\end {array} \right). 
\] 
Although it is simple to find such a solution in this case, we detail the calculations to illustrate the general method. 
The following differential systems give the conditions for reduction at the three levels of the flags.
$$\begin{array}{lll}
(W^{[3]}): &  \left\{ \begin{array}{ccl}
 f'_{{3,1}}   \left( x \right) &=&
\frac{1}{x^{2}}
                              \\
 f'_{{3,2}}   \left( x \right) &=&0
\end{array}\right.
                             \\
 (W^{[2]}): &    \left\{ \begin{array}{ccl}                      
f'_{{2,1}}  \left( x \right) &=&
\,{\frac {1 }{x-1}}  f_{{3,1}}  \left( x \right)+\,{
\frac {1  }{x}}f_{{3,2}}   \left( x \right)+{\frac {2-x}{
{2x}^{2}}}
\end{array}\right.
                               \\
 (W^{[1]}): &         \left\{ \begin{array}{ccl}                          
 f'_{1,1}  \left( x \right) &=&
\,{\frac {1 }{x}}  f_{{2,1}}   \left( x \right)+{\frac {
1-x}{{x}^{2}}}
                               \\
 f'_{1,2}  \left( x \right) &=&
\,{\frac {1  }{x-1}}f_{{2,1}}  \left( x \right)+\frac{3-x}{x^2}.
\end{array}\right.
\end{array}$$

We proceed level by level. 
The first two equations correspond to the first level $W_2^{[3]}$ of the flag.
The condition to remove an element from $W_2^{[3]}$ is that there should be a rational solution to 
the equation $y'= c_{1}.\frac{1}{x^2} + c_{2}.0$. We look for a basis of the $\CC$-vector space of pairs $(c_1,c_2)\in \CC^{2}$ such that there exists   $f\in \mathbf{k}$ with $f' = c_{1}.\frac{1}{x^2} + c_{2}.0$. This space is found to be $2$-dimensional;
for $\underline{c}=(1,0)$, we have $f_{3,1}:=-\frac{1}{x}+c_{3,1}$; 
for $\underline{c}=(0,1)$, we have $f_{3,2}:= c_{3,2}$, where the $c_{3,i}$ are arbitrary constants (their importance will soon be visible). Our gauge transformation is $P^{[3]}= \id + f_{3,1} B_1^{[3]} + f_{3,2} B_2^{[3]} $
and $A^{[2]}:=P^{[3]}[A]$ does not contain any terms from $W_2^{[3]}$.
\par
Now $W_2^{[2]}$ is $1$-dimensional. The equation for the reduction on $W_2^{[2]}$ is 
\begin{eqnarray*}
y'	&=& -\frac{2}{x-1} f_{3,1} - \frac{2}{x} f_{3,2} + \frac{x+1}{x^2}
=-\frac{2}{x-1} \left(-\frac{1}{x}+c_{3,1}\right) - \frac{2}{x} c_{3,2} + \frac{x+1}{x^2}	
\\
	&=& {\frac {-2\,c_{{3,2}}-1}{x}}+{\frac {-2\,c_{{3,1}}+2}{x-1}}+\frac {1}{{x}^{2}}. 
\end{eqnarray*}	
We have necessary and sufficient conditions on the parameters  $c_{3,i}$
to have a rational solution, namely $c_{{3,1}}=1$, $c_{{3,2}}=-\frac12$ 
and then a general rational solution $f_{2,1}:= \frac{-1}{x} + c_{2,1}$.
Our new gauge transformation is 
	$P^{[2]}= \id + (-\frac{1}{x} + c_{2,1}) B_1^{[2]}$
and $A^{[1]}:=P^{[2]}[A]$ does not contain any term from $W_2^{[3]}$
nor from $W_2^{[2]}$.\par
Now we look for pairs $(c_1,c_2)\in \CC^{2}$ such that there exists   $f\in \mathbf{k}$ that is
 a rational solution of 
\begin{eqnarray*}
y' &=&  c_1\left(-\frac{2}{x} f_{2,1} +\frac{2x+1}{x^2} \right) + c_2 \left(-\frac{2}{x-1} f_{2,1} +\frac{2}{x}\right) \\ 
 &=&c_1\left(-\frac{2}{x}  \left(\frac{-1}{x} + c_{2,1}\right) +\frac{2x+1}{x^2} \right) + c_2 \left(-\frac{2}{x-1}  \left(\frac{-1}{x} + c_{2,1}\right) +\frac{2}{x}\right)	\\
 &=& {\frac {-2c_{1}c_{2,1} +2\,c_{{1}}}{x}}+\,{\frac {2c_{{2}}
 \left( -c_{{2,1}}+1 \right) }{x-1}}+{\frac {3c_{{1}}}{{x}^{2}}}.
 \end{eqnarray*}	
This integral is rational if and only if both residues are zero. As the solution $c_1=c_2=0$ is not admissible, we see that a necessary and sufficient condition is $c_{2,1}=1$.
 The set 
of desired pairs $(c_1,c_2)$ is of dimension $2$. For $\underline{c}=(1,0)$, we have $f_{1,1}:=-\frac{3}{x}+c_{1,1}$; 
for $\underline{c}=(0,1)$, we have $f_{1,2}:= c_{1,2}$, where the $c_{1,i}$ are  constants and can be chosen arbitrarily. 
For the simplicity of the expression of the gauge transformation,
we can choose $c_{1,1}=c_{1,2}=0$ (but the other choice is valid too). Our last gauge transformation matrix will be 
$P^{[1]} = \id -\frac3x B_1^{[1]}$. 
\par
Finally, the reduction matrix on $W_2$ is 
\[ P_2:=P^{[3]}P^{[2]}P^{[1]} 
	= \id+\left(-\frac1x +1\right)B_1^{[3]} - \frac12 B_2^{[3]} + \left(-\frac1x +1\right) B_1^{[2]}
	- \frac3x B_1^{[1]} 
\] 
and the matrix $A_2:=P_2[A]$ is reduced on $W_2$. 

\begin{center}\boxed{\textbf{Reduction of the $10$-dimensional block.}}\end{center}
We now turn to the $10$-dimensional block $W_3$. 
The reduction equations are
\[\begin{array}{cl}
(W^{[5]}) : &  \left\{  f_{5,1}' \left( x \right) =0 \right.
\\
(W^{[4]}):  &  \left\{\begin{array}{ccl} 
	 f_{4,1}'(x) &=&
\,{\frac {1  }{x}}f_{5,1}(x)-\frac{1}{2x}
                              \\
	 f_{4,2}'(x) &=&
\,{\frac {1  }{x-1}}f_{5,1}(x)-\frac{1}{2(x-1)}
 	\end{array} \right.
\\
(W^{[3]}):  &  \left\{\begin{array}{ccl}
  f_{3,1}'(x) &=&
\,{\frac {1  }{x}}f_{4,1}(x)+\frac{1}{x}
                      \\
 f_{3,2}'(x) &=&
\,\frac{1}{x}f_{4,2}(x)-\frac{1}{2x}
                         \\
 f_{3,3}'(x) &=&
\,\frac{1}{x-1}f_{4,1}(x) 
                      \\
 f_{3,4}'(x) &=&
\,{\frac {1  }{x-1}}f_{4,2}(x)-\frac{1}{2(x-1)}
	\end{array} \right.
\\
(W^{[2]}):  &  \left\{\begin{array}{ccl}
 f_{2,1}'(x) &=&
\, \frac{1}{x-1}f_{3,1}(x)\,-\frac {1 }{x} f_{3,2}(x)-
	\frac{1}{2(x-1)} 
                              \\
 f_{2,2}'(x) &=& 
\,{-\frac {1 }{x-1}}f_{3,3}(x)+\,{
\frac { 1}{x}}f_{3,4}(x)- \frac{1}{2(x-1)}+ \frac{1}{x^2}
                              \end{array} \right.
                              \\
(W^{[1]}) : &  \left\{
  f_{1,1}'(x) =
\,{\frac {1 }{x-1}} f_{2,1}(x)+\,{
\frac {1 }{x}} f_{2,2}(x) + \frac{2}{x^2}+ \frac{1}{x-1}. 
\right.
\end{array}
\]
The first equation gives $f_{5,1}=c_{5,1}\in \CC$. The equations on $W^{[4]}$ both have rational solutions if and only if $c_{5,1}=\frac12$. We then have $f_{4,i}=c_{4,i}\in \CC$.
\\
Letting $y:=\sum_{i=1}^4 c_i .f_{3,i}$ for unknown $c_i$, the equations on $W^{[3]}$ are 
\[ 
y'  = {\frac {\,c_{{1}}(c_{{4,1}}+1)+\,c_{{2}}(c_{{4,2}}-1/2)}{x}}+\,{
\frac {c_{{3}}c_{{4,1}}+c_{{4}}(c_{{4,2}}-1/2)}{x-1}} .
\]
and we investigate values of $\underline{c}:=(c_1,\ldots,c_4)$ (and $c_{{4,i}}$) for which this may have a rational solution.
Of course, this has a rational solution if and only if both residues are zero. 
The algebraic conditions for both residues to be zero are
\begin{equation}\label{eq1}
\left\{\begin{array}{lll}
c_{{1}}(c_{{4,1}}+1)+\,c_{{2}}(c_{{4,2}}-1/2)&=&0\\
c_{{3}}c_{{4,1}}+c_{{4}}(c_{{4,2}}-1/2)&=&0.
\end{array} \right. \end{equation}
 We view \eqref{eq1} as a linear system in the $\underline{c}$ in coefficients in $\CC (c_{4,i})$. We study for which $c_{4,i}$  the space of  solutions $\underline{c}$ of \eqref{eq1} has maximal dimension. 
Here, it would be $4$ if and only if $c_{{4,1}}+1=c_{{4,1}}=c_{{4,2}}-1/2=0$
 which cannot occur. 
We see that it has dimension $3$ if and only if either 
$c_{{4,1}}+1=c_{{4,2}}-1/2=0$ or $c_{{4,1}}=c_{{4,2}}-1/2=0$.
Then, the only possibilities are 

$\left\{  c_{{3}}=0,  c_{4,1}=-1,c_{{4,2}}=1/2 \right\}$ and 
$ \left\{ c_{{1}}=0, c_{{4,1}}=0,c_{{4,2}}=1/2 \right\} $. We need to study each component separately.
It turns out that both lead to the same result, a reduced form. We show how things go on the second component. The computations for the first component may be found in the Maple Worksheet \cite{DrWe20a} and are detailed in \cite{dreyfus2021differential}.
\\
We have $c_{{4,1}}=0$ and $c_{{4,2}}=\frac{1}{2}$.
The set of $\underline{c}$ for which the equation has a rational solution is a $3$-dimensional $\CC$-vector space;
it is generated by $(0,1,0,0)$, $(0,0,1,0)$ and $(0,0,0,1)$.
We have thus have $f_{3,i}(x)=c_{3,i}\in \CC$ for $i=2,3,4$ and $f_{3,1}$ remains unknown: the equation 
$f_{3,1}'(x)=-\frac1x$ has no rational solution.
\\
So we cannot remove $B_1^{[3]}$ from the result.  However, the constant $c_{3,1}$ will play a role in the reduction process.
\begin{remark}
As  $B_1^{[3]}$ appears in $\glie$, the adjoint action of $\adiag$ implies that $B_1^{[2]}$ and $B_1^{[1]}$
will be present in $\glie$, even if we found transformations which might seem to remove them from the reduced matrix. In the matrices of $\glie$ given at the end of the computation, the third one has
$B_1^{[3]}$ as its off-diagonal part, the fourth one has $B_1^{[2]}$ and the fifth one is  $B_1^{[1]}$.
\end{remark}

Letting $y=c_1 .f_{2,1}+c_2 .f_{2,2}$, the family of reduction equations on $W_2$ is now:
$$ 
y'=\frac{-c_{1}c_{3,2}+c_{2}c_{3,4}}{x}+\frac{c_{1}(c_{3,1}-1/2)-c_{2}(c_{3,3}+1/2)}{x-1}+\frac{c_{2}}{x^{2}}.
$$
The condition for both residues to be zero gives again a linear system on $c_1$ and $c_2$
\[ \left\{ \begin{array}{llllll} 
-&c_{1}c_{3,2} &+&c_{2}c_{3,4}&=&0\\ 
+&c_{1}(c_{3,1}-1/2)  &-&c_{2}(c_{3,3}+1/2) &=&0.
	\end{array}\right.
\]
The space of solutions $(c_1,c_2)$ has maximal dimension $2$ when  $c_{3,2}=c_{3,4}=0$, $c_{3,1}=1/2$
and $c_{3,3}=-\frac{1}{2}$. Now, for $\underline{c}=(1,0)$, we obtain $f_{2,1}=-\frac{1}{x}+c_{2,1}$; 
for $\underline{c}=(0,1)$, we find $f_{2,2}= c_{2,2}$. The last equation is 
$$f_{1,1}'(x)=\frac{1}{x-1}\left(-\frac{1}{x}+c_{2,1}\right) +\frac{c_{2,2}}{x} + \frac{2}{x^2}+\frac{1}{x-1}=\frac{1}{x}-\frac{1}{x-1}+\frac{c_{2,1}}{x-1} +\frac{c_{2,2}}{x} + \frac{2}{x^2}+\frac{1}{x-1}.$$ This 
imposes $c_{2,1}=0$ and $c_{2,2}=-1$ and $f_{1,1}(x)=\frac{1}{x}+c_{1,1}$, where $c_{1,1}$ is a constant that can be chosen arbitrary.
We obtain 
the reduced form
\[ 
A_{\textrm{red}}(x) := \left( \begin {array}{cccc|cccc} 1&0&\frac{1}{x}&0&0&0&0&0
\\   \frac{1}{x-1}&1&0&-\frac{1}{x}&0&0&0&0
\\  0&0&1&0&0&0&0&0\\  0&0& \frac{1}{x-1}&1&0&0&0&0\\ \hline 0&0&0&-\frac{1}{x}&1&0&\frac{1}{x}&0\\  0&0&0&0& \frac{1}{x-1}&1&0&-\frac{1}{x}\\  0&0&0&0&0&0&1&0\\  0&0&0&0&0
&0& \frac{1}{x-1}&1\end {array} \right).
\]
The associated Lie algebra is  spanned by
\begin{eqnarray*}
 {\displaystyle{\tiny \left(\begin {array}{cccc|cccc} 1&0&0&0&0&0&0&0\\  0
&1&0&0&0&0&0&0\\  0&0&1&0&0&0&0&0
\\  0&0&0&1&0&0&0&0\\\hline  0&0&0&0&1&0&0
&0\\  0&0&0&0&0&1&0&0\\  0&0&0&0&0&0
&1&0\\  0&0&0&0&0&0&0&1\end {array} \right) 
, 
\left(
\begin {array}{cccc|cccc} 0&0&0&0&0&0&0&0\\  1&0&0&0&0
&0&0&0\\  0&0&0&0&0&0&0&0\\  0&0&1&0
&0&0&0&0\\\hline    0&0&0&0&0&0&0&0\\  0&0&0
&0&1&0&0&0\\  0&0&0&0&0&0&0&0\\  0&0
&0&0&0&0&1&0\end {array} \right) 
, 
\left(\begin {array}{cccc|cccc} 0&0
&1&0&0&0&0&0\\  0&0&0&-1&0&0&0&0\\  0
&0&0&0&0&0&0&0\\  0&0&0&0&0&0&0&0
\\\hline    0&0&0&-1&0&0&1&0\\  0&0&0&0&0&0&0
&-1\\  0&0&0&0&0&0&0&0\\  0&0&0&0&0&0
&0&0\end {array} \right) }}
, \\ 
{\displaystyle{\tiny 
\left(\begin {array}{cccc|cccc} 0&0&0&0&0&0
&0&0\\  0&0&1&0&0&0&0&0\\  0&0&0&0&0
&0&0&0\\  0&0&0&0&0&0&0&0\\\hline    0&0&1/2
&0&0&0&0&0\\  0&0&0&-1/2&0&0&1&0\\  0
&0&0&0&0&0&0&0\\  0&0&0&0&0&0&0&0\end {array}
 \right) 
 , \left(\begin {array}{cccc|cccc} 0&0&0&0&0&0&0&0
\\  0&0&0&0&0&0&0&0\\  0&0&0&0&0&0&0
&0\\  0&0&0&0&0&0&0&0\\\hline    0&0&0&0&0&0
&0&0\\  0&0&1&0&0&0&0&0\\  0&0&0&0&0
&0&0&0\\  0&0&0&0&0&0&0&0\end {array} \right). }}
\end{eqnarray*}
This gives us the Lie algebra $\mathfrak{g}=\textrm{Lie}( A_{\textrm{red}}(x) )$ of the differential Galois group.
It is $5$-dimensional, whereas the Lie algebra associated to the original matrix $A(x)$ had dimension $14$.
This shows that the Picard-Vessiot extension is obtained from 
$K_{\textrm{diag}}$ by performing only one integral.

\begin{remark}\label{rem4}
We recall that $\gdiag$ is the Lie algebra associated to $\adiag(x)$ and (cf. the proof of Theorem \ref{theo1})
\[
 	\gdiag= \left\{ \left(\begin{array}{c|c} D_1 & 0 \\\hline 0 & D_2 \end{array}\right) 
		\left|  \; \exists \,  S, \; \hbox{such that } \; \left(\begin{array}{c|c} D_1 & 0 \\\hline S & D_2 \end{array}\right) \in \glie
		 \right\} \right. .	
\]
Let us set
\[
 	\glie_s:=\glie/\gdiag = \left\{ \left(\begin{array}{c|c} 0 & 0 \\\hline S & 0 \end{array}\right) 
	 	\left|  \; \exists \, D_1, D_2, \; \hbox{such that } \; \left(\begin{array}{c|c} D_1 & 0 \\\hline S & D_2 \end{array}\right) \in \glie
		\right\} \right. . 	
\]
Note that neither $\gdiag$ nor $\glie_s$ are subalgebras of $\glie$.
Here, $\glie$ has dimension $5$, $\gdiag$ has dimension $4$ and $\glie_{sub}:=\glie \cap \glsub$ has dimension $1$. However, we see that the set $\glie_s$ of ``off-diagonal'' parts of elements of $\glie$ has dimension $3$ so that $\glie \subsetneq \gdiag \oplus \glie_s$. 
Our reduction process computes a subalgebra $\glie$ of $\liealg(A)$ such that its $\glie_s$ has minimal dimension.  
Now,  an odd phenomenon occurs;~in the course of the reduction, an off-diagonal element in $\liealg(A)$ may be "absorbed" as the triangular part of an element that was present.
For example, in the third matrix, an element of $\glsub$ has a coefficient $-\frac1x$ after reduction so it becomes the lower triangular part of the constant matrix from $\liealg{(\adiag)}$ corresponding to $\frac1x$.  
So the minimization of the dimension of $\glie_s$ is a necessary, but a priori not sufficient condition to reduce the system. 
To prove that the reduction process is complete, we need to show that there are no gauge transformation which send the last matrix 
to nilpotent elements whose coefficients are in the Wei-Norman decomposition of $\adiag$.
A simple computation shows that the last element $\gsub$ cannot be ``absorbed'' as the off-diagonal part of an element of $\gdiag$ so our system is indeed in reduced form. 
\end{remark}


\section{Computation of the Reduction Matrix on an Isotypical Flag.}\label{sec4}

Let $\hlie:=\liealg(A)$. As above, we let 
\begin{eqnarray*}
 	\hdiag&:=& \left\{ \left(\begin{array}{c|c} D_1 & 0 \\\hline 0 & D_2 \end{array}\right) 
		\left|  \; \exists \,  S \; \hbox{such that } \; \left(\begin{array}{c|c} D_1 & 0 \\\hline S & D_2 \end{array}\right) \in \hlie
		 \right\} \right. , 	
\\
 	\hlie_s&:=&\hlie/\hdiag = \left\{ \left(\begin{array}{c|c} 0 & 0 \\\hline S & 0 \end{array}\right) 
	 	\left|  \; \exists \, D_1, D_2 \; \hbox{such that } \; \left(\begin{array}{c|c} D_1 & 0 \\\hline S & D_2 \end{array}\right) \in \hlie
		\right\} \right.
		\, \textrm{ and } \\
	\hsub&:=&  \left\{ \left(\begin{array}{c|c} 0 & 0 \\\hline S & 0 \end{array}\right) \in \hlie\right\}.
\end{eqnarray*}
We have $\hdiag=\liealg(\adiag)$. Neither $\hdiag$ nor $\hlie_s$ are subalgebras of $\hlie$. We have $\hlie \subset \hdiag \oplus \hlie_s$ but the nilpotent example,  see $\S \ref{nilpotent-example-continued}$,  shows that the inclusion may be strict. Our reduction strategy will consist of three steps.

\begin{itemize}
\item[\fbox{$1$}] Find a gauge transformation $P=\id + B$, where\footnote{Here, it would actually be enough to take $B\in  \hlie_s(\mathbf{k})$ but our choice simplifies both the exposition and the implementation for a minor additional cost.} $B\in  \glsub(\mathbf{k})$, and 
$\tilde{\hlie}:=\liealg(P[A])$ such that $\tilde{\hlie}_s$ has minimal dimension. The result will depend on parameters. 
\item[\fbox{$2$}] Look for an eventual gauge transformation $P=\id + B$, where $B\in  \glsub(\mathbf{k})$ which,  maps each element of $\tilde{\hlie}_{\textrm{sub}}$  to an element of $\tilde{\hlie}_{\textrm{sub}}$ whose coefficients are in the Wei-Norman decomposition of $\adiag$. This depends again on parameters.
\item[\fbox{$3$}] Compute, with a Groebner basis, conditions on the remaining parameters to have a Lie algebra $\liealg(P[A])$ of minimal dimension. 
\end{itemize}
Step \fbox{$1$} is the main part of the algorithm. It consists in trying to eliminate as many generators as possible in $\hlie_s$. 
Heuristically, Step \fbox{$1$} seems to be always sufficient to obtain a reduced form. However, Steps \fbox{$2$} and \fbox{$3$} are necessary to have a mathematically guaranteed procedure. 
\\
We use the isotypical decomposition 
$  \glsub   = \bigoplus_{i=1}^{\kappa}  W_{i} $
of Proposition~\ref{propo3}. 
Proposition~\ref{propo2} tells us that in the reduction process, we may (and will)  perform a reduction on each isotypical block $W_{i}$ independently.
\\
We consider an isotypical block  ${W = V_1 \oplus \dots \oplus V_{\nu}}$ where the $V_i$ are indecomposable pairwise
isomorphic $\Psi$-spaces.
We follow the construction above Definition \ref{isotypical-flag} to obtain a $\Psi$-isotypical flag on $W$:
\[
     W= W^{[\mu]} \supsetneq W^{[\mu-1]}\supsetneq \cdots \supsetneq W^{[1]} \supsetneq W^{[0]}=\{0\},
     \]
with 
$W^{[k]}=\displaystylee \bigoplus_{j=1}^{\nu_{k}} V_{j}^{[k]}$
and we have a $\Psi$-isomorphism 
$\phi_j: V_{1}^{[k]}\rightarrow V_{j}^{[k]}$.  Recall that $W^{[k]}/W^{[k-1]}$ is a direct sum of pairwise isomorphic irreducible $\Psi$-spaces.
Before we continue, we need to enrich our toolbox with the following fundamental classical algorithm.

\subsection{Differential systems with a parametrized right-hand side}
We recall a classical computational lemma on rational solutions of differential systems with a parametrized right-hand side.
We reprove it here for self-containedness though it is well known to specialists.

\begin{lemma}[\cite{Si91a,Ba99a,Be02a}]\label{exp-parametrebis} 
Let $m\in \mathbb{N}^{*}$. Given a matrix $\Lambda(x) \in\CM_m(\mathbf{k})$ and $t$ 
vectors $\vec{b}_i(x) \in \mathbf{k}^m$, we consider the differential system
with parametrized right-hand-side 
	$$ Y'(x) = \Lambda(x) Y(x) + \sum_{i=1}^t s_i \vec{b}_i(x),$$
where the $s_i$ are scalar parameters. \\
The set of tuples $\left(F(x), (c_1,\ldots,c_t)\right) \in \mathbf{k}^m \times \CC^t$, such that the differential system 
${ Y'(x) = \Lambda(x) Y(x) + \sum_{i=1}^t c_i \vec{b}_i(x)}$ admits a rational solution $Y(x)=F(x)$ is a finite-dimensional $\CC$-vector space. Furthermore, with our assumptions on $\mathbf{k}$,  one can effectively compute a basis of this vector space.
\end{lemma}

\begin{proof} 
We give a short proof of this well known fact, from \cite{Be02a}, page 889.\\
A vector $F(x)=(f_1 (x),\ldots,f_m (x))^T\in \mathbf{k}^m$  is a rational solution of the differential system ${ Y'(x) = \Lambda(x) Y(x) + \sum_{i=1}^t c_i \vec{b}_i(x) }$, for given constants $c_i$, if and only if the vector  ${Z(x):=(f_1 (x),\ldots,f_m (x),c_1,\ldots,c_t)^T}$ is a rational solution of the homogeneous first order system
	$$ Z'(x) = \left( \begin{array}{c|c} \Lambda(x) & \vec{b}_1(x),\ldots,\vec{b}_t(x) \\ \hline 0 & 0 \end{array}\right) Z(x).$$
The rational solutions of the latter form a $\CC$-vector space which, by Assumption \cercle{2}, see $\S\ref{base-field}$, can be explicitly computed.
\end{proof}

\begin{remark}\label{rem2}
Regarding the proof of Lemma \ref{exp-parametrebis}, we see that we may replace $\mathbf{k}$ by any field 
which satisfies Assumption \cercle{2} of Section \ref{base-field}. 
By \cite{Si91a}, Lemma 3.5, we may thus replace 
the base field $\mathbf{k}$ by $\mathbf{k}(t_1,\dots,t_{\ell})$ where $t_{i}'=0$ and the new constant field
$\CC (t_1,\dots,t_{\ell})$ is a transcendental extension of $\CC$.
\end{remark}

%
\subsection{Reduction on a level $W^{[k]}/W^{[k-1]}$ of the  isotypical block $W$} \label{reduction-one-level}
Let
us fix $k\in \{ 1,\dots, \mu\}$. Assume that we have performed reductions on the levels $W^{[\mu]}/W^{[\mu-1]},\dots,W^{[k+1]}/W^{[k]}$ (this assumption being void if $k=\mu$) and that we 
want to perform reduction on the level $W^{[k]}/W^{[k-1]}$. Our matrix $A$ has thus been transformed into a matrix $A^{[k]}$.
{As in the nilpotent example, see \S \ref{nilpotent-example-continued}}, the reduction on the previous levels may have introduced a set $T_k$ of parameters in the 
off-diagonal
coefficients of this matrix $A^{[k]}$, and 
an affine variety $\CT_k$ defined by the
 algebraic conditions satisfied by these parameters. 
So, at this stage
the matrix $A^{[k]}(x,\underline{t})$ 
has off-diagonal coefficients 
in $\mathbf{k}(T_k)$ with the constraint $\underline{t}\in \CT_k$. The construction below will show how $T_k$, $\CT_k$, and $A^{[k]}$
are built with a decreasing recursion from $T_{\mu}=\varnothing$, $\CT_{\mu}=\varnothing$, and $A^{[\mu]}=A$. 
The matrix $\Psi$ of the adjoint action of $\adiag$ is unchanged at each step and does not depend on the parameters.
\\\par

Let $B_{1,1}, \ldots, B_{r,1}$ be a basis of $V_{1}^{[k]}/V_{1}^{[k-1]}$. It induces a basis of  $V_{1}^{[k]}/V_{1}^{[k-1]}\otimes_{\CC}\CC(T_k )$. 
The isomorphism ${\phi_j:V_1^{[k]}\rightarrow V_j^{[k]}}$ induces an isomorphism 
${\tilde{\phi}_j:V_1^{[k]}/V_1^{[k-1]}  \rightarrow V_j^{[k]}/V_j^{[k-1]} }$ 
	so we may define 
$B_{i,j}:=\tilde{\phi}_j(B_{i,1})$ to obtain an adapted basis of 
 $V_j^{[k]}/V_j^{[k-1]}$ and, hence, of $V_j^{[k]}/V_j^{[k-1]}\otimes_{\CC}\CC(T_k )$.
In this basis, the restriction of $\Psi$ to each $V_j^{[k]}/V_j^{[k-1]}\otimes_{\CC}\CC(T_k )$ will have the same matrix, which we call 
$\Lambda^{[k]}$.
This matrix $\Lambda^{[k]}$ has coefficients in $\mathbf{k}$.

Let $\Psi^{[k]} := \Psi|_{W^{[k]}/W^{[k-1]}}$ be the 
restriction of $\Psi$ to $W^{[k]}/W^{[k-1]}$. We still call $\Psi^{[k]}$ the restriction to $(W^{[k]}/W^{[k-1]})\otimes_{\CC}\mathbf{k}(T_{k})$.
In our adapted basis of $W^{[k]}/W^{[k-1]}$, the matrix of $\Psi^{[k]}$ is block diagonal with all blocks equal to $\Lambda^{[k]}$. 

For any matrix $B \in  W^{[k]}$, we have 

	\begin{equation}\label{eq4}
	 \Psi(B) = \Psi^{[k]} (B) + \widetilde{B},
	\quad \textrm{ with } \; \widetilde{B}\in  W^{[k-1]}.
	\end{equation}

Simplifying notations, let us set $\nu=\nu_{k}$. We decompose the matrix $A^{[k]}$ of our system at this stage as 
\[A^{[k]}(x,\underline{t}) = \overline{A}(x,\underline{t}) + \displaystylee 
 	\sum_{j=1}^{\nu} \left(  \sum_{i=1}^r   a_{i,j}(x,\underline{t}) B_{i,j}  \right).
	\] 
The coefficients $a_{i,j}(x,\underline{t})$ are in  $\mathbf{k}(T_k)$; the matrix
$\overline{A}(x,\underline{t})$  represents the remaining components of parts of $A^{[k]}$, 
including
the components $W^{[\ell]}/W^{[\ell-1]}\otimes_{\CC}\mathbf{k}(T_{k})$ with $\ell\neq k$.

We look for a gauge transformation of the form 
\begin{equation}\label{jauge}
{
	P(x,\underline{t}) = \id_n + B 
	\, \textrm{ where } \, B= \displaystylee
	\sum_{j=1}^{\nu} \left(  \sum_{i=1}^r   f_{i,j}(x,\underline{t}) B_{i,j}  \right), 
}
\end{equation}
with $f_{i,j}(x,\underline{t})\in\mathbf{k}(T_{k})$. 
We apply 
Propositions~\ref{propo2} and \eqref{eq4}  
to obtain the existence of 
$\widetilde{B}_{i,j}\in W^{[k-1]}$ 
{(the part of $\Psi(B)$ which is sent to $W^{[k-1]}$)}
such that 

\begin{multline}\label{eq10}
P[A^{[k]}] = \left[\overline{A}(x,\underline{t}) 
	+
\sum_{j=1}^{\nu} \left(  \sum_{i=1}^r   f_{i,j}(x,\underline{t}) \widetilde{B}_{i,j} \right) \right ]\\
{
+
\underbrace{
\sum_{j=1}^{\nu} \sum_{i=1}^r   \left(  a_{i,j}(x,\underline{t}) B_{i,j}
+  f_{i,j}(x,\underline{t})\Psi^{[k]}( {B}_{i,j} ) 
- f_{i,j}'(x,\underline{t}) {B}_{i,j}
\right)}_{\textrm{components of }P[A^{[k]}] \textrm{ on }W^{[k]}/W^{[k-1]}}
.}
\end{multline}
Suppose we 
hoped
to remove all the $B_{i,j}$. We would have to remove all of the second sum in \eqref{eq10}. 
For each $j\in \{ 1 ,\dots, \nu \}$, let
\[ 
	\vec{Y}_j:=\left(\begin{matrix} f_{1,j} (x,\underline{t})\\ \vdots \\ f_{r,j}(x,\underline{t}) \end{matrix}\right)
	 \; \text{ and } \; 
	 \vec{b}_j  :=\left(\begin{matrix}a_{1,j}(x,\underline{t}) \\ \vdots \\a_{r,j}(x,\underline{t}) \end{matrix}\right).
\]
The elimination conditions would become 
$$
\left\{ \begin{matrix} 
	\vec{Y}_1' & = & \Lambda^{[k]} \vec{Y}_1 &\ + & \vec{b}_1 \\
		& \vdots & & & \\
	\vec{Y}_\nu' & = & \Lambda^{[k]} \vec{Y}_\nu & + & \vec{b}_\nu	.
\end{matrix}\right.
$$
However, some of these systems may have no rational solution  
whereas some combination of the  $\vec{Y}_i$ could be rational and lead to (partial) reduction.
Indeed, by performing reduction, we are trying to eliminate irreducible $\Psi$-subspaces of 
$W^{[k]}/W^{[k-1]}\otimes_{\CC}\CC (T_k)$. By Goursat's Lemma, see Lemma \ref{lem1}, these are of the form
$\{\sum_{i=1}^{\nu}c_{j} \tilde{\phi}_i(v), \;v\in V_{1}^{[k]}/V_{1}^{[k-1]} \otimes_{\CC}\CC (T_k)  \}$ with $c_{i}\in \CC(T_{k})$.
In order to perform reduction, we thus need to look for constants
$\underline{c}=(c_1,\ldots, c_\nu)$ with $c_i \in \CC(T_{k})$ such that the following system has a nonzero rational solution in $\mathbf{k}(T_k)$:
\begin{equation} \label{condition-reduction}
	\vec{Y}' = \Lambda^{[k]} \vec{Y} + \sum_{i=1}^{\nu} c_i \vec{b}_i.
\end{equation}
Note that $\Lambda^{[k]}$ does not depend on the parameters and 
Lemma \ref{exp-parametrebis} stays valid with the field $\mathbf{k}$ replaced by $\mathbf{k}(T_k)$, see Remark \ref{rem2}.  
To decide when the system has nonzero rational solutions, one first finds bounds on valuations at the poles and at infinity (this is possible because $\Lambda^{[k]}$ does not depend on the parameters);  this reduces the problem to solving a system of linear equations whose right hand side depends linearly on the parameters $c_i$. The compatibility conditions for this system (obtained, for example, by gaussian elimination) yield
a matrix $M(\underline{t})$
with coefficients in $\CC(T_k)$ so that the system (\ref{condition-reduction}) has a rational solution if and only if
\begin{equation} \label{rang-minimal} M(\underline{t}).  \left(\begin{matrix} c_1 \\ \vdots \\ c_{\nu} \end{matrix}\right) = 0.
\end{equation}
We want  $M(\underline{t})$ with  $\underline{t}\in \CT_k$ to have a kernel of maximal dimension, as this kernel 
allows us to compute 
irreducible $\Psi$-subspaces
that can be removed in the reduction process. 
Let $\CV_k$ denote the algebraic conditions on $T_k$ which encode the fact that 
$M(\underline{t})$ with  $\underline{t}\in \CT_k$ has minimal rank 
 $\nu-d$; this can be computed for example with a Groebner basis, see \cite{CoLiOS07a}. 
We set $\CT_{k-1}:=\CT_k \cap \CV_k$. 
Now $\CT_{k-1}$ is a finite union of irreducible algebraic varieties; this can again be computed with a Groebner basis \cite{CoLiOS07a}. For each of these irreducible varieties, we  proceed as follows.
Applying these conditions to the matrix $ M(\underline{t}) $, we choose a basis 
$\CB$ of $\ker( M(\underline{t}) )$ in $\CC( T_{k})^{\nu}$. 
For $\underline{c}_j$ in $\CB$, we compute the corresponding general rational solution $\vec{F}_j$ to the system 
$ Y'=\Lambda^{[k]} Y + \sum_{i=1}^{\nu} c_{i,j} \vec{b}_i $. Note that $[\Lambda^{[k]}]$
may have rational solutions (this was the case in our nilpotent example,  see $\S \ref{nilpotent-example-continued}$) hence the need for a general solution. 
Note that at this stage it may be necessary to introduce additional
parameters in order to express this general solution, in which case we add these new
additional parameters to $T_k$ to form $T_{k-1}$.
\begin{remark} As we saw in the nilpotent example $\S \ref{nilpotent-example-continued}$, 
the process of 
passing from $\CT_{k}$ to $\CT_{k-1}$
may add constraints that fix the value of a constant, thus withdrawing 
it from later computations.
 
\end{remark}

Now we have found a new adapted basis $ \overline{B}_{i,j}(\underline{t})$ of $W^{[k]}/W^{[k-1]}\otimes_{\CC} \CC(T_{k})$ and a gauge transformation 
\[ P^{[k]}(x,\underline{t}) = \id_n + \displaystylee
	\sum_{j=1}^d \left(  \sum_{i=1}^r   f_{i,j}(x,\underline{t}) \overline{B}_{i,j}(\underline{t})  \right).
\]
Applying this gauge transformation will remove from $W^{[k]}/W^{[k-1]}$ the 
	$\Psi$-spaces	
spanned by the
$\overline{B}_{i,j}(\underline{t})$ with $j\leq d$. 
Because of the condition on minimality of the rank of $M$ and the nature of the 
$\Psi$-subspaces
of $W^{[k]}/W^{[k-1]}\otimes_{\CC} \CC(T_{k})$, 
no other matrix in $W^{[k]}/W^{[k-1]}\otimes_{\CC} \CC(T_{k})$ can be removed using a gauge transformation as in (\ref{jauge}).

 \subsection{The Full Reduction} \label{full-reduction}
We now perform the reduction on the whole isotypical block $W$ with its isotypical flag
\[
     W= W^{[\mu]} \supsetneq W^{[\mu-1]}\supsetneq \cdots \supsetneq W^{[1]} \supsetneq W^{[0]}=\{0\}.
     \]
\begin{trivlist}
\item \textbf{Step 1. }    
We start from a set of parameters $T_{\mu}=\varnothing$ and algebraic conditions $\CT_{\mu}=\varnothing$. 
The matrix of the system is $A^{[\mu]}:=A$. We perform the reduction process of Section 
\ref{reduction-one-level} on $W^{[\mu]}/W^{[\mu-1]}$; we obtain a gauge transformation $P^{[\mu]}$
and $A^{[\mu-1]}:= P^{[\mu]}( A^{[\mu]} )$. If $\mu>1$, we go down to $W^{[\mu - 1]}$ and iterate 
until the level $W^{[1]}$.
The complete gauge transformation used for the successive reductions on $W$ is 
$ P_{W} (x,\underline{t}):=\displaystyle\prod_{k=1}^{\mu} P^{[k]}(x,\underline{t}) $.
It contains a set $T_W:=T_{0}$ of parameters $t_i$, subject to the set $\CT_W:=\CT_{0}$ of algebraic conditions. Note that, by construction, the matrices $P^{[k]}$ all commute pairwise so the product 
$P_{W} (x,\underline{t})$ is well defined.  
The same remark will hold for the reduction matrix of Theorem \ref{theoreme-reduction} below.\par 
\item \textbf{Step 2. } As explained in Remark \ref{rem4},  in the course of this reduction some ``off-diagonal'' element of the Lie algebra may be ``absorbed'' by turning a diagonal element of $\gdiag$ into a triangular one. Let $g_{1}(x),\dots, g_{\delta}(x) \in \mathbf{k}$ be the $\CC$ linearly independent elements appearing in the Wei-Norman decomposition of $\adiag (x)$. Let $B_{1},\dots,B_{k}$ be a basis of $\mathfrak{h}_{W}$, the $\Psi$-space obtained after this step of reduction process. Let $A_{W}(x,\underline{t}):=P_{W} (x,\underline{t})[A]$.  We have to compute the set of $f_{i}\in \mathbf{k}(T_W )$, $C_{1},\dots,C_{\delta}\in\CM_n (\CC(T_W ))$ such that $\widetilde{P}_{W}[A_{W}]=A_{W}+C_{1}g_1 +\dots +C_{\delta}g_{\delta}$, where $\widetilde{P}_{W}(x) = \id_n + \displaystylee
	 \sum_{i=1}^{k}   f_{i} B_{i}$. By Proposition \ref{propo2}, this is equivalent to solving an inhomogeneous linear differential equation in the same form as the one in Lemma \ref{exp-parametrebis} in the field $\mathbf{k}(T_W )$. This provides a new set of parameters that we must add to $T_W$ and additional algebraic constraints $\CT_W$.
Using again a Groebner basis, compute an element $\underline{t}_0\in \CT_W$ such that 
$Lie( \widetilde{P}_{W}(x,\underline{t}_0)[A_W (x,\underline{t}_0)] )$ has minimal dimension (this is a rank optimization computation).
Finally, set ${P_{W} (x):=\widetilde{P}_{W} (x,\underline{t}_0)}$. 
\end{trivlist}

\begin{theorem}\label{theoreme-reduction}
For each isotypical block $W_i$ in the isotypical decomposition $\glsub=\displaystylee \bigoplus_{i=1}^{\kappa} W_{i}$,  let $P_{W_{i}} (x)$ denote the (partial) reduction matrix constructed in the above paragraph. 
Now let $P (x):=\prod_{i=1}^{\kappa} P_{W_{i}} (x)$ 
and $A_{\textrm{red}}(x):=P (x)[A (x)]$.
Then the system $[A_{\textrm{red}}(x)]$ is in reduced form and $P(x)$ is the corresponding reduction matrix.
\end{theorem}

\begin{remark}
In many situations, like the first three examples, no parameters are required to reduce the system. In that case, no Groebner bases are needed and only linear algebra is used in the reduction process, making the algorithm quite effective.
\end{remark}

\begin{proof}
In virtue of Theorem \ref{theo1},  we deduce that there exists a gauge transformation
${Q \in \Big\{\id_n+B(x), B(x)\in \glsub\left( \mathbf{k}\right)\Big\}}$,
such that $[Q[A_{red}]]$ is in reduced form. We have $Q[A_{red}]=QP[A]$, where $QP \in \Big\{\id_n+B(x), B(x)\in \glsub\left( \mathbf{k}\right)\Big\}$.
Let $\hlie:=\liealg(A_{red})$. As above, for $\star\in \{\glie,\hlie\}$, we let 
\begin{eqnarray*}
 	\star_{diag}&:=& \left\{ \left(\begin{array}{c|c} D_1 & 0 \\\hline 0 & D_2 \end{array}\right) 
		\left|  \; \exists \,  S \; \hbox{such that } \; \left(\begin{array}{c|c} D_1 & 0 \\\hline S & D_2 \end{array}\right) \in \star
		 \right\} \right. , 	
\\
 	\star_s&:=&  \left\{ \left(\begin{array}{c|c} 0 & 0 \\\hline S & 0 \end{array}\right) 
	 	\left|  \; \exists \, D_1, D_2 \; \hbox{such that } \; \left(\begin{array}{c|c} D_1 & 0 \\\hline S & D_2 \end{array}\right) \in \star
		\right\} \right.
		\, \textrm{ and } \\
	\star_{sub}&:=&  \left\{ \left(\begin{array}{c|c} 0 & 0 \\\hline S & 0 \end{array}\right) \in \star\right\}.
\end{eqnarray*}
Since $[Q[A_{red}]]$ is in reduced form, its Lie algebra is $\glie$. By Remark \ref{rem3}, $\glie \subset \hlie$. Then, $\glie_{s} \subset \hlie_{s}$. By construction, the gauge transformation $P$ minimize $\hlie_{s}$, so $\hlie_{s} \subset \glie_{s}$ and we have 
$\hlie_{s} = \glie_{s}$. By Proposition \ref{propo2}, $\hlie_{diag} = \glie_{diag}$.  
Now,  as explained in Remark \ref{rem4}, we still could have a strict inclusion $\glie \subsetneq \hlie$. 
If that were the case,
the condition $\glie \subsetneq \hlie$ would imply  by minimality of $\glie_{s}$, that a gauge transformation in $\hsub(\mathbf{k})$ would transform the coefficient of an off-diagonal element into one that is present in the Wei-Norman decomposition of $\adiag$ (the "absorption" mechanism described in Remark \ref{rem4}, turning a diagonal element of $\hlie$ into a triangular one).
By the second minimality condition,  see Step 2 above,  this phenomenon does not occur and we conclude that $\glie = \hlie$ and
$[A_{red}]$ is in reduced form. 
\end{proof}

\section{A General Algorithm for Reducing a Differential System}\label{sec:algo}

We now have tools to put general linear differential systems,  i.e.  those whose diagonal part may have more than two diagonal blocks, into reduced form.

\subsection{An Iteration Lemma}
In order to iterate the reduction process of $\S \ref{sec4}$ to block triangular systems, we need the following lemma. 

\begin{lemma} \label{lemme:iteration}
Let $n_{1},n_{2},n_{3}\in \mathbb{N}^{*}$, and for $i\in \{1,2,3\}$, let $A_{i}\in \CM_{n_{i}}(\mathbf{k})$. Assume that the differential systems with respective matrices (in what follows, $S$ is an $n_3 \times n_2$ matrix with coefficients in $\mathbf{k}$)
 $$\left(\begin{array}{c|c|c}
A_{1}&0&0\\ \hline
0 &A_{2}& 0 \\ \hline
0 &0 &A_{3}
 \end{array}\right)
 \; \mathrm{ and }\; 
 \left(\begin{array}{c|c}
A_{2}& 0 \\ \hline
S&A_{3}
 \end{array}\right)
 $$
are in reduced form. Then, letting
$$A:= \left(\begin{array}{c|c|c}
A_{1}&0&0\\ \hline
0 &A_{2}& 0 \\ \hline
0 &S&A_{3}
 \end{array}\right),$$
 the system $[ A ]$ is in reduced form. 
\end{lemma}

\begin{proof}
By our first assumption, we may apply the reduction process of $\S \ref{sec4}$, see Theorem~\ref{theoreme-reduction}, to the system with matrix
 $$\left(\begin{array}{cc|c}

		A_{1}	& 	0&0 \\ 
		0		&	A_{2}&0\\ \hline 
		0		&	S&A_{3}
	 \end{array} 
	\right).
$$
Let $n:=n_{1}+n_{2}+n_{3}$. Due to Theorem~\ref{theo1}, there exists a reduction matrix of the form 
{
\[ P:= 
\left(\begin{array}{c|c|c} \id_{n_{1}} & 0 & 0 \\ \hline 0 & \id_{n_{2}} & 0 \\ \hline P_{1} & P_2& \id_{n_{3}}  \end{array}\right)
 \in \CM_{n}(\mathbf{k}).
 \]
 }
We then find 
\begin{eqnarray*}
\Psi \left(\begin{array}{c|c}(P_{1} & P_2)\end{array}\right)
&=&
{-\begin{array}{c|c}(P_{1} & P_2)\end{array}\left(\begin{array}{c|c}
A_{1}&0\\ \hline
0 &A_{2}
 \end{array}\right)+A_{3}\begin{array}{c|c}(P_{1} & P_2)\end{array}}
 \\ & = &
\begin{array}{c|c}( {A_{3}P_{1}-P_{1}A_{1} } &
{A_{3}P_{2}-P_{2}A_{2}}).\end{array}
\end{eqnarray*}

With Proposition \ref{propo2}, we find that 
{
\[
P[A]=\left(\begin{array}{c|c|c} A_{1}& 0 & 0 \\ \hline 0 & A_{2} & 0 \\ \hline 
A_{3}P_{1}-P_{1}A_{1}  - P'_{1} &
S +   A_{3}P_{2}-P_{2}A_{2} - P'_{2} & A_{3} \end{array}\right).
\]
}
Since the latter is reduced, we know that $\liealg(P[A])=\mathfrak{g}$, where $\mathfrak{g}$ is the Lie algebra of the differential Galois group of $[A]$.
By  Remark~\ref{rem3}, we find ${\liealg (P[A])=\mathfrak{g}\subset\liealg (A)}$. By construction of $\liealg (A)$, 
any 
	matrix in $\liealg (A)$
must have the form ${ \left(\begin{array}{c|c|c} \star & 0 & 0 \\ \hline 0 & \star & 0 \\ \hline 0 & \star &\star \end{array}\right)}$. As $\mathfrak{g}\subset\liealg (A)$, the same holds for $\mathfrak{g}$.
Since $P_1$ acts only on the bottom left block of $P[A]$, we thus find that without loss of generality, we may assume $P_{1}=0$.
Then we see that the reduction matrix will have no effect on the $A_1$ block but will only act on the block 
 $\boxed{ \begin{array}{c|c}
		A_{2}	& 	0 \\ \hline
		S	&	A_{3}
	 \end{array} }.
 $
 As the latter is in reduced form, we find, see Proposition~\ref{propo1}, that for all $P_{2}$, the Lie algebra of 
 
 $$
\left(\begin{array}{c|c}   \id_{n_{2}} & 0 \\ \hline P_{2} & \id_{n_{3}}  \end{array}\right)\left[\left(\begin{array}{c|c}
		A_{2}	& 	0 \\ \hline
		S	&	A_{3}
	 \end{array}\right)\right]
	 =
	 \left(\begin{array}{c|c}
		A_{2}	& 	0 \\ \hline
		S +  {A_{3}P_{2}-P_{2}A_{2}}   - P'_{2}  
			&	A_{3} \end{array}\right)
$$ contains the Lie algebra of 
	$\left(\begin{array}{c|c}
		A_{2}	& 	0 \\ \hline
		S	&	A_{3}
	 \end{array}\right)$. It is now clear that we have the inclusion
	 ${\liealg(A)\subset \liealg (P[A])=\glie}$. By Remark \ref{rem3} $\glie\subset \liealg(A)$ and we find 
	 $\liealg(A)=\mathfrak{g}$, i.e. $[A]$ is in reduced form.
\end{proof}

\subsection{The Algorithm}
Let us now describe the global reduction process. Let ${\mathcal{A}(x)\in \CM_{n}(\mathbf{k})}$ and consider the linear differential system ${Y'(x)=\mathcal{A}(x)Y(x)}$.
The contribution of this paper to this general algorithm is part $(4)$ below.
\begin{enumerate}
\item \label{step:factor}
Factor the linear differential system, see e.~g. \cite{CoWe04a,Ba07a,Ho07b} and references therein. 
We obtain a matrix $A(x)\in \CM_{n}(\mathbf{k})$ such that the system $Y'(x)=\mathcal{A}(x)Y(x)$ is equivalent to $Y'(x)=A(x)Y(x)$, where 
$$A(x)= \left(\begin{array}{c|c|c|c}
A_{1}(x)&&&0\\ \hline
&\ddots&&\\ \hline
 & S_{i,j}(x) &\ddots& \\ \hline
& &&A_{k}(x)
 \end{array}\right),$$
and each diagonal block $ Y'(x)=A_{\ell}(x)Y(x)$, 
	$\ell=1\dots k$, 
is irreducible.

\item \label{step:diagonal}
Using for example \cite{BaClDiWe16a,ApCoWe13a},
compute a reduced form of the block-diagonal system 
$$ Y'(x)= \left(\begin{array}{c|c|c}
A_{1}(x)&&0\\ \hline
&\ddots&\\ \hline
0&&A_{k}(x)
 \end{array}\right)Y(x).$$ 
Note that the reduced form 
	$\left(\begin{array}{c|c|c}
A_{1,red}(x)&&0\\ \hline
&\ddots&\\ \hline
0&&A_{k,red}(x)
 \end{array}\right)$ 
 may have entries in a finite algebraic extension $\mathbf{k_{0}}$ of $\mathbf{k}$. 
 Let $P_{diag}(x)\in \mathrm{GL}_{n}(\mathbf{k_{0}})$ be the corresponding gauge transformation.
 \item \label{step:3} Compute $$A_{diag,red}(x):=P_{diag}(x)[A(x)]=\left(\begin{array}{c|c|c|c}
A_{1,red}(x)&&&0\\ \hline
 &\ddots&&\\ \hline
\vdots & \ddots &A_{k-1,red}(x)& \\ \hline
\mathfrak{S}_{k,1}(x)& \hdots &\mathfrak{S}_{k,k-1}(x)&A_{k,red}(x)
 \end{array}\right) .$$
\item \label{step:4}
Let $ {\mathcal A}_k := A_{k,red}$ and $\ell:=k$.
\\
While $\ell\geq 2$ do 
	\begin{enumerate}
	\item \label{step:iteration}
	Apply the reduction process of $\S \ref{sec4}$, see Theorem \ref{theoreme-reduction} with $\mathbf{k}$ replaced by $\mathbf{k_{0}}$,  see Remark \ref{rem5},  to compute a reduced form of 
		$$\left(\begin{array}{c|c}
			A_{\ell-1,red }(x)		&	0\\ \hline
			\mathfrak{S}_{\ell-1}(x)&  \mathcal{A}_{\ell}(x)
		 \end{array}\right),\hbox{ where }\mathfrak{S}_{\ell-1}(x) := \begin{pmatrix}
		\mathfrak{S}_{\ell,\ell-1}(x)\\ 
		\vdots \\
		\mathfrak{S}_{k,\ell-1}(x)
		\end{pmatrix},
		$$
 is the block column below $A_{\ell-1,red}(x)$ in $A_{diag,red}(x)$.
	\item 
		Let $\mathcal{A}_{\ell-1}(x)$ be this new reduced form. Let $\ell:=\ell-1$ and iterate.
	\end{enumerate}
End do.
\end{enumerate}
The correctness of Step~\ref{step:iteration}
is ensured by Lemma~\ref{lemme:iteration}. It follows that the resulting system $Y'(x)=\mathcal{A}_{1}(x)Y(x)$ is a reduced form of $Y'(x)=A(x)Y(x)$. 

\begin{remark}\label{rem5}
In Step \ref{step:iteration},  we may have to introduce an algebraic extension $\mathbf{k_{0}}$ of $\C(x)$ and compute  solutions in $\mathbf{k_{0}}$ of linear differential systems with  coefficients in $\mathbf{k_{0}}$ and a parameterized right hand side. This can be reduced (see \cite{Si91a}) to computing solutions in $\mathbf{k}$ of a  system of bigger dimension. From $\S \ref{base-field}$, we see that it is still possible. Although it would require some extra work, it would not be a practical obstacle.
\end{remark}

\begin{example}
When $k=3$, Step \ref{step:4} performs the following. Consider the matrix given by Step \ref{step:3}.
$$A_{diag,red}(x)=\left(\begin{array}{c|c|c}
A_{1,red}(x)&0&0\\ \hline
\mathfrak{S}_{2,1}(x) &A_{2,red}(x)&0 \\ \hline
\mathfrak{S}_{3,1}(x)&\mathfrak{S}_{3,2}(x)&A_{3,red}(x)
 \end{array}\right).$$ 
 We start by reducing the matrix 
 $\left(\begin{array}{c|c}
			A_{2,red }(x)		&	0\\ \hline
			\mathfrak{S}_{3,2}(x)&  A_{3,red}(x)
		 \end{array}\right)$ 
to obtain $\mathcal{A}_{2}(x)$. Then we reduce 
		$\left(\begin{array}{c|c}
A_{1,red}(x)&0\\ \hline
\mathfrak{S}_{3}(x) &\mathcal{A}_{2}(x)
 \end{array}\right)$, with $\mathfrak{S}_{3}(x) := \begin{pmatrix}
		\mathfrak{S}_{2,1}(x)\\ 
		\mathfrak{S}_{3,1}(x)
		\end{pmatrix}$, 
to obtain the final reduced matrix.
\end{example}

\section{Computation of the Lie algebra of the differential Galois group}\label{sec:lie}
We consider
a differential system $[\mathcal{A}(x)] : \,Y'(x)=\mathcal{A}(x)Y(x)$ with $\mathcal{A}(x)\in \mathcal{M}_{n}(\mathbf{k})$.  In this section, we review how to compute the Lie algebra $\glie$ of the differential Galois group $G$. 
We first assume that $[\mathcal{A}(x)]$ is in reduced form; the non-reduced case is addressed in Remark~\ref{rem1}. 
We stress the fact that, from now on, all the results are mostly well known and are included for completeness.
\\
We can find a  Wei-Norman decomposition of $\mathcal{A}(x)$. 
We compute the smallest $\CC$-vector space containing its generators and stable under the Lie bracket. Let $B_{1},\dots, B_{\sigma}$ be a basis of this space. We know that the smallest \emph{algebraic} Lie algebra containing the $B_{i}$ is $\glie$. \par 

An algorithm for computing the smallest algebraic Lie algebra containing the $B_i$ can be found in \cite{FieGra}. In order to be self-contained, we are going to summarize this work.\par 
 It follows from \cite{Chev}, Chapter II, Theorem 14, that the Lie algebra generated by a finite family of algebraic Lie algebras is algebraic. Therefore, to compute $\glie$ it is sufficient to be able to compute 
 {
 $\glie_{i}:= \liealg(B_{i})$  for all $i \in \{ 1,\dots, \sigma \}$.
 }\par 
To be able to  compute $\glie$, we are thus reduced to the  following problem: {given} 
a matrix ${B\in  \CM_{n}(\CC)}$, compute  $\liealg(B)$.
Let $B=D+N$  be the additive Jordan decomposition of $B$, where $D$ is diagonalizable, $N$ is nilpotent, and $DN=ND$. From \cite{Chev}, Chapter II, Theorem 10, we deduce that $$\liealg(B)=\liealg(D)\oplus \liealg(N).$$
Let us compute $\liealg(N)$. The matrix $N$ is nilpotent. As we can see in \cite{Chev}, Chapter~II, $\S 13$, Proposition 1, the $\CC$-vector space spanned by $N$ is an algebraic Lie algebra, with corresponding algebraic group 
$\{\exp (\alpha N), \alpha\in \CC\}$, which is a vector group.
Therefore, $$\liealg(N)=\mathrm{Vect}_{\CC}(N).$$
 Let us compute $\liealg (PD_{0}P^{-1})$, where $PD_{0}P^{-1}=D$, $P$ is an invertible matrix, and ${D_{0}=\mathrm{Diag}(d_{1},\dots,d_{n})}$ is a diagonal matrix. Set $$\Delta := \Big\{(e_{1},\dots, e_{n})\in \mathbb{Z}^{n}\Big| \sum_{\ell=1}^{n} e_{\ell}d_{\ell}=0 \Big\}.$$
By Chevalley, see for instance \cite{FieGra}, Theorem 2, we obtain 
  $$\liealg(D_{0})=\left\{\mathrm{Diag}(a_{1},\dots,a_{n})\in \CC^{n}\Big| \sum_{\ell=1}^{n} e_{\ell}a_{\ell}=0, \forall (e_{1},\dots, e_{n})\in \Delta \right\},$$
  and
 $$\liealg(D)=P\liealg (D_{0})P^{-1}.$$
\begin{remark}\label{rem1}
If we start from a system $[\mathcal{A}(x)]$ which is not in reduced form, the algorithm presented in $\S \ref{sec:algo}$ will compute a finite field extension $\mathbf{k}_0$ of $\mathbf{k}$ and a matrix ${\mathcal{A}_{red}(x)\in \CM_{n}(\mathbf{k}_{0})}$ such that $[\mathcal{A}_{red}(x)]$ is a reduced form of $[\mathcal{A}(x)]$.
\\
 Let $G_{\mathbf{k}_0}$ be the differential Galois group over $\mathbf{k}_0$. Since the gauge transformation that performs the reduction has entries in $\mathbf{k}_{0}$, and the Galois group is invariant under gauge transformation, $G_{\mathbf{k}_0}$ is  the differential Galois group of $[\mathcal{A}_{red}(x)]$ over $\mathbf{k}_0$. By Lemma~32 in  \cite{ApCoWe13a}, $G_{\mathbf{k}_0}$ is connected.
Note that by the Galois correspondence, see \cite{PS03}, Proposition 1.34, $G/G_{\mathbf{k}_0}$ is finite, which means that $G_{\mathbf{k}_0}$ is the connected component of the identity of $G$. So the Lie algebras of $G$ and $G_{\mathbf{k}_0}$ coincide and we may apply the above construction to obtain 
$\liealg(\mathcal{A}_{red}(x))$ and hence $\glie$.
\end{remark}

\section{Computation of the differential Galois group of a reduced form}\label{sec7}
Let
$\mathcal{A}(x)\in \mathcal{M}_{n}(\mathbf{k})$; let $G$ be the differential Galois group of $[\mathcal{A}(x)]$ and $\glie$ be the Lie algebra. We now know, using $\S \ref{sec:lie}$, how to compute $\glie$.
The goal of this section is to explain how, theoretically, one may recover $G$ from $\glie$ when $[\mathcal{A}(x)]$ { is in reduced form}. 
The problem of recovering a connected group from its Lie algebra is solved in \cite{Gr09a}. In this section, we propose a solution, based on ideas from \cite[Section 3]{DeJeKo05a}, but we do not claim originality nor algorithmic efficiency in what follows; this section is included for completeness.
\par 
 Since $[\mathcal{A}(x)]$ is in reduced form, by Lemma 32 in  \cite{ApCoWe13a},  we obtain that $G$ is connected.
Let $B_{1},\dots,B_{\sigma}$ be a basis of the $\CC$-vector space $\glie$. 
As $G$ is connected, it is the smallest algebraic group containing $\exp(\glie)$.
It follows that 
	\[G=\overline{\langle \exp(B_{1}),\dots,\exp(B_{\sigma}) \rangle }\] 
it is the smallest algebraic group that contains the matrices $\exp(B_{i})$. 
\par 
So let us compute $\overline{\langle \exp(B_{1}),\dots,\exp(B_{\sigma}) \rangle }$. This problem has been solved in full generality in \cite[Section 3]{DeJeKo05a}.
{It is simpler here, since the algebraic group we are looking for is connected. }We start by a classical observation taken from \cite[Section 3.1]{DeJeKo05a}.\par 
\begin{lemma}[\cite{DeJeKo05a}, Section 3.1]
\label{lemsuperbien}
Let $V_{1}$, $V_{2}$ be affine varieties over $\CC$ and $\psi :V_{1}\rightarrow V_{2}$ be a morphism of affine varieties. 
Let $X\subset V_{1}$ be a Zariski closed subset. 
If we have generators for the vanishing ideal $\mathfrak{f}\subset \CC[V_{1}]$ of $X$, 
then we may 
compute $\overline{\langle \psi (X)\rangle}$.
\end{lemma}

\begin{proof} For self-containedness, we reproduce the proof from \cite[Section 3.1]{DeJeKo05a}.
The morphism  $\psi :V_{1}\rightarrow V_{2}$ corresponds to a homomorphism  $\psi^{\star} :\CC[V_{2}]\rightarrow \CC[V_{1}]$ of the coordinate rings 
{(see \cite{CoLiOS07a}, Proposition 8 in Chapter 4])}. 
Given  generators of the vanishing ideal $\mathfrak{f}\subset \CC[V_{1}]$ of $X$, one can compute  generators of the ideal $(\psi^{\star})^{-1}(\mathfrak{f})$ using a Groebner basis. The latter are the generators of $\overline{\langle \psi (X)\rangle}$.
\end{proof}

We begin by computing the Zariski closure of the group generated by a single matrix $M:=\exp(B)$, with $B\in \glie$.
We have a Dunford decomposition $B=S+N$ with $S$ semi-simple, $N$ nilpotent and $[S,N]=0$.
So $\exp(B)=D\cdot U$ with $D:=\exp(S)$ diagonalizable and $U:=\exp(N)$ unipotent.
 
As $[D,U]=0$, we find that ${\overline{\langle M \rangle }=\overline{\overline{\langle D \rangle } .\overline{\langle U \rangle }}}$. Using Lemma \ref{lemsuperbien}, if we are able to compute $\overline{\langle D \rangle }$ and $\overline{\langle U \rangle }$, we see that we may compute $\overline{\langle M \rangle }=\overline{\overline{\langle D \rangle } .\overline{\langle U \rangle }}$.
Thus, what is left for us to do is to treat the cases where $M$ is unipotent or diagonalizable. \par 

We start with the unipotent case. As $N$ is nilpotent, the map 
\[ \psi :   t\mapsto \exp(t N ) \]
is an algebraic map from $V_1:=\CC$ to $V_2:=\mathrm{GL}_n (\CC)$; moreover,
 $ \exp(t N )$ is a linear combination of a finite number of powers of $N$. Pick a matrix $M=(x_{i,j})$ of indeterminates 
 and eliminate $t$ from the equations $M-\exp(t N )=0$; this makes sense because $\exp(t N )$ is polynomial in $t$.
As shown in Lemma~\ref{lemsuperbien}, this allows us to recover the Zariski closure of the image of $\CC$ under $\psi$, which is $\overline{\langle U \rangle }$.
\par 

We now treat the diagonalizable case. We have $D=\exp(S)$ with $S$ semi-simple.
We may diagonalize $S$ so that, letting 
\[ \mathcal{D} := \mathrm{Diag}(\lambda_{1},\ldots , \lambda_{n}), \quad D_0:=\exp(\mathcal{D})=\mathrm{Diag}(d_{1},\ldots ,d_{n})
\textrm{ with } d_i=\exp(\lambda_i), \]
we have $S= P \mathcal{D}  P^{-1}$
and $D= P  D_0   P^{-1}$ for some $P\in \mathrm{GL}_{n}(\CC)$.
In order to compute  $\overline{\langle D \rangle }$,  it is sufficient to understand the algebraic relations between the eigenvalues of $D$. This will be done in the same way as in the computation of the Lie algebra of a diagonal matrix. 
As we see in \cite{DeJeKo05a}, {Section 3.3.}, the ideal that generates $\overline{\langle D_{0} \rangle }$ will be obtained from the $e_{1}, \ldots,e_{n}\in \CC$ such that $(d_{1})^{e_{1}}\cdots (d_{n})^{e_{n}}=1$ with $d_i=\exp(\lambda_i)$. So, we set 
\[ \Delta' := \Big\{(e_{1},\dots, e_{n})\in \mathbb{Z}^{n}\Big| \sum_{\ell=1}^{n} e_{\ell} \lambda_{\ell}=0 \Big\}, \]
and we find
\[ \overline{\langle D_{0} \rangle }=\Big\{ %
	(d_{1},\dots, d_{n})\in (\CC^*)^{n}\Big| %
	\prod_{\ell=1}^{n} (d_{\ell})^{e_{\ell}}=1, \forall (e_{1},\dots, e_{n})\in \Delta' \Big\}. \]
As the map $X\mapsto P   X   P^{-1}$ is algebraic, Lemma \ref{lemsuperbien} tells us that we may  compute $\overline{\langle D \rangle }$ from the relation 
\[ \overline{\langle D \rangle }=P \overline{\langle D_{0} \rangle }  P^{-1}. \]

Now that we know how to compute the $\overline{\langle \exp(B_{i}) \rangle }$, for every $i \in\{ 1, \dots,\sigma \}$,
Lemma \ref{lemsuperbien} shows that we may compute $\overline{\langle \exp(B_{1}),\dots,\exp(B_{\sigma}) \rangle }=G$
and we are done.\\ 
\par

\begin{remark}
If we start from a system $[\mathcal{A}(x)]$ which is not  in reduced form, the algorithm presented in $\S \ref{sec:algo}$ will compute a finite field extension $\mathbf{k}_0$ of $\mathbf{k}$ and a matrix ${\mathcal{A}_{red}(x)\in \CM_{n}(\mathbf{k}_{0})}$ such that $[\mathcal{A}_{red}(x)]$ is a reduced form of $[\mathcal{A}(x)]$. By Remark \ref{rem1}, the differential Galois groups over $\mathbf{k}_0$ of $[\mathcal{A}_{red}(x)]$ and $[\mathcal{A}(x)]$ coincide and are equal  to $G^{\circ}$, the connected component of the identity of $G$. 
The defining ideal of $G^{\circ}$ gives the algebraic relations over $\mathbf{k}_0$ inside the Picard-Vessiot extension.  
\end{remark}

\begin{example}
{
We have given in $\S \ref{nilpotent-example-continued}$, the generators of the Lie algebra of the nilpotent example. 
A Zariski-dense subgroup of the differential Galois group is generated by $t_{1}\mathrm{Id}_{8}$ with $t_1\in \CC^*$, and the family of the following matrices, with  $t_{2},\dots,t_{5}\in \CC$:}
\begin{tiny}
\begin{eqnarray*}
 \left( \begin {array}{cccc|cccc} 1&0&t_{{2}}&0&0&0&0&0
\\ 0&1&0&-t_{{2}}&0&0&0&0\\ 0&0&1&0
&0&0&0&0\\0&0&0&1&0&0&0&0\\\hline t_{{2
}}&0&0&0&1&0&t_{{2}}&0\\ 0&-t_{{2}}&0&0&0&1&0&-t_{{2
}}\\ 0&0&-t_{{2}}&0&0&0&1&0\\ 0&0&0
&t_{{2}}&0&0&0&1\end {array} \right) ,  &
\left( \begin {array}
{cccc|cccc} 1&0&0&0&0&0&0&0\\ 0&1&t_{{3}}&0&0&0&0&0
\\ 0&0&1&0&0&0&0&0\\ 0&0&0&1&0&0&0
&0\\\hline 0&0&0&0&1&0&0&0\\ t_{{3}}&0&0
&0&0&1&t_{{3}}&0\\ 0&0&0&0&0&0&1&0
\\0&0&-t_{{3}}&0&0&0&0&1\end {array} \right) ,\\
\displaystyle
 \left( \begin {array}{cccc|cccc} 1&0&0&0&0&0&0&0\\ 2
\,t_{{4}}&1&0&0&0&0&0&0\\ 0&0&1&0&0&0&0&0
\\ 0&0&2\,t_{{4}}&1&0&0&0&0\\\hline -t_{
{4}}&0&0&0&1&0&0&0\\0&t_{{4}}&0&0&2\,t_{{4}}&1&0&0
\\ 0&0&-t_{{4}}&0&0&0&1&0\\ 0&0&0&
t_{{4}}&0&0&2\,t_{{4}}&1\end {array} \right) ,& \left( \begin {array}
{cccc|cccc} 1&0&0&0&0&0&0&0\\ 0&1&0&0&0&0&0&0
\\ 0&0&1&0&0&0&0&0\\0&0&0&1&0&0&0
&0\\ \hline  0&0&0&0&1&0&0&0\\ 0&0&t_{{5}}
&0&0&1&0&0\\ 0&0&0&0&0&0&1&0\\ 0&0
&0&0&0&0&0&1\end {array} \right). 
\end{eqnarray*}
\end{tiny}Using the above procedure, one may recover the equations of the Galois group from the data of these generators.
\end{example}

\bibliographystyle{amsalpha} 
\bibliography{ADW_MRS_biblio}
\end{document}